\documentclass[12pt, twoside]{article}
\usepackage{amsmath,amsthm,amssymb}
\usepackage{times}
\usepackage{enumerate}
\usepackage{mathrsfs}
\usepackage{amsfonts}
\usepackage{amssymb}
\usepackage{amssymb}
\usepackage{amssymb}
\usepackage{amssymb}
\usepackage{amssymb}
\usepackage{amssymb}
\usepackage{amssymb}
\usepackage{amssymb}
\usepackage{amsmath}
\usepackage{mathdots}
\usepackage[all]{xy}

\def\titlerunning#1{\gdef\titrun{#1}}
\makeatletter
\def\author#1{\gdef\autrun{\def\and{\unskip, }#1}\gdef\@author{#1}}
\def\address#1{{\def\and{\\\hspace*{18pt}}\renewcommand{\thefootnote}{}%
\footnote {#1}}%
\markboth{\autrun}{\titrun}}
\makeatother
\def\email#1{e-mail: #1}
\def\subjclass#1{{\renewcommand{\thefootnote}{}%
\footnote{\emph{Mathematics Subject Classification (2010):} #1}}}
\def\keywords#1{\par\medskip
\noindent\textbf{Keywords.} #1}

    \newcommand{\CO}{{\mathcal {O}}} 
     
    \newcommand{\CS}{{\mathcal {S}}} 
     
    \newcommand{\CW}{{\mathcal {W}}}

    \newcommand{\lenth}{{\mathrm {\lenth}}}

\newtheorem{thm}{Theorem}[section]

\newtheorem{lem}[thm]{Lemma}

\newtheorem{proposition}[thm]{Proposition}

\theoremstyle{definition}
\newtheorem{defin}[thm]{Definition}

\newtheorem*{xrem}{Remark}

\numberwithin{equation}{section}

\begin{document}

%%%%% To ease editing, add:

\baselineskip=17pt

%%%%%%%%%%%%%%%%

%% In the running head, give an abbreviation of the title.
\titlerunning{Bessel functions and local converse conjecture of Jacquet}

\title{Bessel functions and local converse conjecture of Jacquet}

\author{Jingsong Chai}

\date{}

\maketitle

\address{Jingsong Chai: College of Mathematics and Econometrics, Hunan University, Changsha, 410082,
China; \email{chaijingsong@hnu.edu.cn}}


%%%%%%%%

\begin{abstract}
In this paper, we prove the local converse conjecture of Jacquet
over $p$-adic fields for $GL_n$ using Bessel functions.

%% Keywords are optional
\end{abstract}

\section{Introduction}
Let $F$ be a p-adic field, and $\psi$ be a nontrivial additive
character of $F$ with conductor precisely $\CO$, the ring of
integers of $F$. Let $\pi$ be an irreducible admissible generic
representation of $GL_n(F)$. If $\rho$ is an irreducible admissible
generic representation of $GL_r(F)$, one can attach an important
invariant local gamma factor $\gamma(s,\pi\times \rho,\psi)$ via the
theory of  local Rankin-Selberg integrals by Jacquet,
Piatetski-Shapiro and Shalika (\cite{JPSS}). This invariant can also
be defined by Langlands-Shahidi method (\cite{Sha}).

These invariants can be used to determine the representation $\pi$
up to isomorphism. In \cite{Hen}, Henniart proved that, for
irreducible admissible generic representations $\pi_1$ and $\pi_2$
of $GL_n(F)$, if the family of invariants $\gamma(s,\pi_1\times
\rho,\psi)=\gamma(s,\pi_2\times \rho, \psi)$, for all $r$ with
$1\leq r\leq n-1$, and for all irreducible admissible generic
representations $\rho$ of $GL_r(F)$, then $\pi_1\cong \pi_2$. This
is then strengthened by J.Chen (\cite{Ch}) and Cogdell and
Piatetski-Shapiro (\cite{CPS}) by decreasing $r$ from $n-1$ to
$n-2$. Such type of results together with global version first
appeared in \cite{JaL} for $GL(2)$ and \cite{JPSS1979} for $GL(3)$.
A general conjecture, which is due to Jacquet can be formulated as
follows.

\textbf{Conjecture 1 (Jacquet).} Assume $n\ge 2$. Let $\pi_1$ and
$\pi_2$ be irreducible generic smooth representations of $GL_n(F)$.
Suppose for any integer $r$, with $1\leq r \leq [\frac{n}{2}]$, and
any irreducible generic smooth representation $\rho$ of $G_r$, we
have

$$\gamma(s, \pi_1\times \rho, \psi )= \gamma(s, \pi_2\times \rho,
\psi ),$$ then $\pi_1$ and $\pi_2$ are isomorphic.

In the present paper, we will prove \textbf{Conjecture 1} using
Bessel functions. The author was recently informed that this
conjecture has also been proved by H.Jacquet and Baiying Liu
independently using a different method, see \cite{JL}.

By the work of Dihua Jiang, Chufeng Nien and Shaun Stevens in
section 2.4, \cite{JNS}, this conjecture has been reduced to the
following conjecture when both $\pi_1,\pi_2$ are unitarizable
irreducible supercuspidal representations.

\textbf{Conjecture 2.} Assume $n\ge 2$. Let $\pi_1$ and $\pi_2$ be
irreducible unitarizable and supercuspidal smooth representations of
$GL_n(F)$. Suppose for any integer $r$, with $1\leq r \leq
[\frac{n}{2}]$, and any irreducible generic smooth representation
$\rho$ of $G_r$, we have

$$\gamma(s, \pi_1\times \rho, \psi )= \gamma(s, \pi_2\times \rho,
\psi ),$$ then $\pi_1$ and $\pi_2$ are isomorphic.

It is this conjecture that we will prove in this paper. The main
result can be stated as follows.

\begin{thm}
\textbf{Conjecture 2} is true, and so is \textbf{Conjecture 1}.
\end{thm}

A crucial ingredient in the proof is Bessel function, which has its
own interests. Given an irreducible admissible generic
representation $\pi$ of $GL_n(F)$, one can attach a Bessel function
$j_{\pi}$ to $\pi$. Such functions were first defined over p-adic
fields by D.Soudry in \cite{Soudry:1984} for $GL_2(F)$, and then
were generalized to $GL_n(F)$ by E.M.Baruch in \cite{Baruch:2005},
to other split groups by E.Lapid and Zhengyu Mao in
\cite{LapidMao:2013}, respectively. A general philosophy is that the
local gamma factors $\gamma(s,\pi\times \rho,\psi)$ are intimately
related to the Bessel functions of $\pi,\rho$. In many cases, we
know that local gamma factors can be expressed as certain Mellin
transform of Bessel functions, see for example
\cite{CPS98,Shahidi:2002,Soudry:1984}. This expression is the
starting point of proving stability of local gamma factors, which is
crucial in applying converse theorem to Langlands functoriality
problems. For the case at hand, such Mellin transform is also
expected but has not yet been proved according to the author's
knowledge. However, it is still possible to derive an equality of
Bessel functions from equalities of local gamma factors via local
Rankin-Selberg integrals. Then the above conjecture follows as the
Bessel function $j_\pi$ determines the representation $\pi$ up to
isomorphism by the weak kernel formula (Theorem 4.2 in \cite{cjs}).
This is the basic idea of our proof.

There are much progress made towards this conjecture in recent
years. In particular, Chufeng Nien in \cite{Nien} proved an analogue
of this conjecture in finite field case. Dihua Jiang, Chufeng Nien
and Shaun Stevens in \cite{JNS} formulated an approach using
constructions of supercuspidal representations to attack this
conjecture in p-adic case, and proved it in many cases, including
the cases when $\pi_1,\pi_2$ are supercuspidal representations of
depth zero. Later on based on this approach, Moshe Adrian, Baiying
Liu, Shaun Stenvens and Peng Xu in \cite{ALSX} proved the conjecture
for $GL_n(F)$ when $n$ is prime. There are also some work on similar
problems for other groups. See \cite{Ba1995, Ba1997, JS, Zh1, Zh2}
for examples. For a more comprehensive survey on local converse
problems and related results, see relevant sections in \cite{Jiang,
JN}.

%One important ingredient of the above work is the following (non-disjoint) decomposition.
%Let $\alpha=\begin{pmatrix} & I_{n-1} \\ 1 &   \end{pmatrix}$, then
%\[
%GL_n(F)= \cup_{ 0\le i\le [\frac{n}{2}], n-[\frac{n}{2}]\le j\le n } N_n\alpha^iP_n \alpha^j N_n
%\]
%where $P_n$ is the mirabolic subgroup and $N_n$ is the subgroup of
%upper triangular unipotent matrices. This decomposition was first
%used by Chufeng Nien in her work \cite{Nien}, and also plays an
%important role in \cite{JNS}. \\

An important ingredient of the approach suggested in \cite{JNS} is
to reduce the conjecture to show the existence of certain Whittaker
functions (called \textit{special pair of Whittaker functions}) for
a pair of unitarizable supercuspidal representations $\pi_1,\pi_2$
of $GL_n(F)$. In \cite{JNS} and \cite{ALSX}, such Whittaker
functions were found in many cases using the constructions of
supercuspidal representations.

We will explain our proof in more details, and the above ingredient
is also important. To prove \textbf{Conjecture 2}, as explained
above it suffices to show that, under the assumptions of the
conjecture, the unitarizable supercuspidal representations
$\pi_1,\pi_2$ have the same Bessel functions. By Proposition 5.3 in
\cite{cjs}, it is reduced to show that the normalized Howe vectors
(see Definition 3.1 below), which are certain partial Bessel
functions in the Whittaker models providing nice approximations to
Bessel functions, satisfy
$W^1_{v_m}(a\omega_n)=W^2_{v_m}(a\omega_n)$ for any diagonal matrix
$a$, where $\omega_n$ is the longest Weyl element. For this purpose,
when $n=2r+1$ is odd, we consider the following Rankin-Selberg
integrals on $GL_{2r+1}(F)\times GL_r(F)$ (for the unexplained
notations, see section 2 for details)
\[
\gamma(s,\pi^*_i\times \rho,\psi^{-1})\int_{N_r\backslash G_r}
\int_{M(r\times r)} \widetilde{W}^i_{\omega_n.v_m} \left(
\begin{pmatrix} g&&   \\ x&I_{r}& \\ &&1   \end{pmatrix}
\begin{pmatrix} \omega_{2r} & \\ & 1 \end{pmatrix} \alpha^{r+1}a \right)W'(g)|det(g)|^{s-\frac{r+1}{2}}dxdg
\]
\[
=\omega_{\rho}(-1)^{r-1}\int
\widetilde{\widetilde{W^i}}_{\omega_n.v_m}\left(
\begin{pmatrix} g& \\ &I_{r+1} \end{pmatrix}\omega_{n,r} \begin{pmatrix} \omega_{2r} & \\ & 1 \end{pmatrix}
\alpha^{r+1}a^{-1}
\right)\widetilde{W'}(g)|det(g)|^{1-s-\frac{r+1}{2}}dg,
\]
where $\rho$ is any generic irreducible smooth representation of
$GL_r(F)$.

By Lemma 2.3 below, it suffices to show
\[
\widetilde{W}^1_{\omega_n.v_m} \left(
\begin{pmatrix} g&&   \\ x&I_{r}& \\ &&1   \end{pmatrix}
\begin{pmatrix} \omega_{2r} & \\ & 1 \end{pmatrix} \alpha^{r+1}a \right)
=\widetilde{W}^2_{\omega_n.v_m} \left(
\begin{pmatrix} g&&   \\ x&I_{r}& \\ &&1   \end{pmatrix}
\begin{pmatrix} \omega_{2r} & \\ & 1 \end{pmatrix} \alpha^{r+1}a \right)
\]
on certain open dense subset of the domain in the integrals.
Inspired by the work of \cite{JNS}, we will first in section 3 show
that, the normalized Howe vectors satisfy certain properties similar
to \textit{special pair of Whittaker functions} in a slightly weak
form on certain Bruhat cells. Then combining with the work of Jeff
Chen in \cite{Ch}, these properties will imply the above identities,
which finishes the proof in the odd case. The even case can then be
deduced from the odd case.

We finally remark that Nien used Bessel functions in her's proof of
finite field analogue (\cite{Nien}), and in this paper we give the
first Bessel function proof of \textbf{Conjecture 1} over p-adic
fields.

The paper is organized as follows. In section 2, we recall some
backgrounds and preparations on Rankin-Selberg integrals and Bessel
functions. We then study Howe vectors in detail in section 3. In the
last section, we prove the theorem.

\section{Preparations}

Use $G_n$ to denote $GL_n(F)$, and embed $G_{n-1}$ into $G_n$ on the
left upper corner. Let $N_n$ be subgroup of the upper triangular
unipotent matrices. $A_n$ the subgroup of diagonal matrices. Use
$P_n$ to denote the mirabolic subgroup consisting of matrices with
the last row $(0,...,0,1)$. We extend the additive character $\psi$
to $N_n$, still denoted as $\psi$, by setting
\[
\psi(u)=\psi(\sum_{i=1}^{n-1}u_{i,i+1}) \ \ \ \ u=(u_{ij})\in N_n.
\]

If $\pi$ is an irreducible admissible generic representation of
$G_n$, use $\CW(\pi,\psi)$ to denote the Whittaker model of $\pi$
with respect to $\psi$. If $W\in \CW(\pi,\psi)$, define
\[
\widetilde{W}(g):=W(\omega_n\cdot {^tg^{-1}})
\]
where $\omega_n=\begin{pmatrix}  & &  1\\ &\iddots & \\ 1 & &
\end{pmatrix}$, then the space of functions
\[
\{\widetilde{W}(g): W\in \CW(\pi,\psi) \}
\]
is the Whittaker model of $\pi^*$ with respect to $\psi^{-1}$, where
$\pi^*$ denotes the contragredient of $\pi$.

Suppose $\pi$ and $\pi'$ are irreducible admissible generic
representations of $G_n$ and $G_r$ respectively, with associated
Whittaker models $\mathcal{W}(\pi,\psi)$ and $\mathcal{W}(\pi',
\psi^{-1})$. For our purpose, we will assume $r<n$. For any $W\in
\mathcal{W}(\pi,\psi)$, $W'\in \mathcal{W}(\pi', \psi^{-1})$, $s\in
\mathbb{C}$ a complex number, and any integer $j$ with $n-r-1\ge j
\ge 0$, set $k=n-r-1-j$, and let $$I(s, W, W',
j)=\int_{N_r\backslash G_r}\int_{M(j\times r)} W\begin{pmatrix} g &
0 & 0 \\ x & I_j & 0 \\ 0 & 0 & I_{k+1}
\end{pmatrix} W'(g)|det g|^{s-(n-r)/2} dxdg,$$
where $M(j\times r)$ denotes the space of matrices of size $j\times
r$.

We have the following basic result in the theory of local
Rankin-Selberg integrals.

\begin{thm}(\cite{JPSS}) For any pair of generic
irreducible admissible representations $\pi$ and $\pi'$ on $G_n$ and
$G_r$, we have:

(1). The integrals $I(s, W, W', j)$ converge absolutely for $Re(s)$
large;

(2). The integrals $I(s, W, W', j)$ span a fractional ideal in
$\mathbb{C}[q^s, q^{-s}]$ with a unique generator $L(s, \pi\times
\pi')$ such that $L(s, \pi\times \pi')$  has the form
$P(q^{-s})^{-1}$  for some polynomial $P\in \mathbb{C}[x]$ with
$P(0)=1$;

(3). There exists a meromorphic function $\gamma(s, \pi\times \pi',
\psi)$, independent of the choices of $W$, $W'$, such that
$$I(1-s,
{\pi^*}(\omega_{n,r})\widetilde{W}, \widetilde{W'}, k) =
\omega_{\pi'}(-1)^{n-1}\gamma(s, \pi \times \pi', \psi)I(s, W, W',
j),$$ where $n-r-1\ge j \ge 0$, $k=n-r-1-j$,
$\omega_{n,r}=\begin{pmatrix} I_{r} & \\ & \omega_{n-r}
\end{pmatrix}$ and $\omega_{\pi'}$ is the central character of
$\pi'$.

\end{thm}

%\hspace{14cm}  $\Box$ \\

\begin{xrem}
The meromorphic function $\gamma(s, \pi\times \pi', \psi)$ is called
the \emph{local $\gamma$-factor} of $\pi$ and $\pi'$.
\end{xrem}

Now let $\pi$ be a generic irreducible unitarizable representation
of $G_n$. Consider the space of functions
\[
\overline{\CW}:=\{\overline{W}(g): W(g)\in \CW(\pi,\psi) \}
\]
where $'\bar \ \ '$ denotes the complex conjugate.

Then with the right translation by $G_n$, $\overline{\CW}$ is an
irreducible representation of $G_n$, with
$\overline{W}(ug)=\psi^{-1}(u)\overline{W}(g)$ for $u\in N_n,g\in
G_n$. Thus $\overline{\CW}$ is the Whittaker model with respect to
$\psi^{-1}$ of some generic irreducible unitarizable representation
$\tau$ of $G_n$.

Since $\pi$ is unitarizable, we have a $G_n$-invariant inner product
$(W_1,W_2)$ on some complex Hilbert space. Then for $\overline{W}\in
\overline{\CW}$, view it as a smooth linear functional on
$\CW(\pi,\psi)$ via the above form

%\[
%(W_1, W_2)=\int_{N_nZ_n\backslash G_n} W_1(g)\overline{W}_2(g)dg \ \
%\ \ \ \ (*)
%\]
%is a nonzero $G_n$-invariant inner product on $\CW(\pi,\psi)$ (see
%also Corollary 3.4 in (\cite{LM})).

\[
l_{\overline{W}}: W_v(g)\to (W_v, \overline{W}).
\]
Since $(*)$ is $G_n$-invariant, this gives an isomorphism between
the representation $\tau$ on $\overline{\CW}$ and $\pi^*$. We record
it as a proposition.

\begin{proposition}
\label{prop2} If the representation $\pi$ is generic irreducible
unitarizable with Whittaker model $\CW(\pi,\psi)$, then the
representation $(\tau,\overline{\CW})$ is a Whittaker model
of $\pi^*$ with respect to $\psi^{-1}$. \\
\end{proposition}

We also need the following lemma, which is Corollary 2.1 in
\cite{Ch}.

\begin{lem}
\label{le2.3}Let $H$ be a complex smooth function on $G_r$
satisfying

$$H(ug)=\psi(u)H(g)$$
for any $g\in G_r$, $u\in N_r$.

If for any irreducible generic smooth representation $\rho$ of
$G_r$, and for any Whittaker function $W\in
\mathcal{W}(\rho,\psi^{-1})$, the function defined by the following
integral vanishes for $Re(s)$ large

$$\int_{N_r\backslash G_r} H(g)W(g)|det(g)|^{s-k}dg=0,$$
where $k$ is some fixed constant, then $H\equiv 0$.
\end{lem}

\textbf{Remark.} We don't need to require the function $H$ to be in
the Whittaker model of some generic representation as in \cite{Ch}.
The proof there works for general Whittaker functions.   \\

Next we introduce the Bessel functions briefly. For more details see
\cite{Baruch:2005, cjs}. Let $(\pi,V)$ be a generic irreducible
representation of $G_n$. Take $W_v\in \mathcal{W}(\pi, \psi)$.
Consider the integral for $g\in N_nA_n\omega_n N_n$

\[
\int_{N_n}W_v(gu)\psi^{-1}(u)du.
\]
This integral stabilizes along any compact open filtration of $N_n$,
and the map
\[
v\to \int_{N_n}^*W_v(gu)\psi^{-1}(u)du,
\]
where $\int^*$ denotes the stabilized integral, defines a Whittaker
functional on $V$. By the uniqueness of Whittaker functional, there
exists a scalar $j_{\pi,\psi}(g)$ such that

\[
\int_{N_n}^* W_v(gu)\psi^{-1}(u)du=j_{\pi,\psi}(g)W_v(I).
\]

\begin{defin}
\label{def2} The assignment $g\to j_\pi(g)=j_{\pi,\psi}$ defines a
function on $N_nA_n\omega_n N_n$, which is called the \emph{Bessel
function} of $\pi$ attached to $\omega_n$.
\end{defin}

We extend $j_{\pi}$ to $G_n$ by putting $j_{\pi}(g)=0$ if $g\notin
N_nA_n\omega_nN_n$, and still use $j_{\pi}$ to denote it and call it
the Bessel function of $\pi$. It is easy to check that $j_{\pi}$ is
locally constant on $N_nA_n\omega_n N_n$(See Theorem 1.7 and remarks
above it in \cite{Baruch:2005}) and
$j_{\pi}(u_1gu_2)=\psi(u_1)\psi(u_2)j_{\pi}(g)$ for any $u_1, u_2\in
N_n, g\in G_n$.

A property of Bessel function which is important to us is the
following weak kernel formula, see Theorem 4.2 in \cite{cjs}. \\

\begin{thm}
\label{thm1} (Weak Kernel Formula) For any $b\omega_n$,
$b=diag(b_1,...,b_n)\in A_n$, and any $W\in \mathcal{W}$, we have

$$W(b\omega_n)=$$
$$\int j_{\pi}\left(b\omega_n \begin{pmatrix} a_1 & \\ x_{21} & a_2 \\ & & \ddots \\
x_{n-1,1} & \cdots & x_{n-1,n-2} & a_{n-1} \\ & & & &  1
\end{pmatrix}^{-1} \right) W\begin{pmatrix} a_1 & \\ x_{21} & a_2 \\
& & \ddots \\ x_{n-1,1} & \cdots & x_{n-1,n-2} & a_{n-1}\\ & & & & 1
\end{pmatrix} $$
$$|a_1|^{-(n-1)}da_1|a_2|^{-(n-2)}dx_{21}da_2\cdots |a_{n-1}|^{-1}dx_{n-1,1}\cdots dx_{n-1,n-2}da_{n-1}, $$
where the right side is an iterated integral, $a_i$ is integrated
over $F^{\times}\subset F$ for $i=1,...,n-1$, $x_{ij}$ is integrated
over $F$ for all relevant $i,j$, and all measures are additive
self-dual Haar measures on $F$.
\end{thm}

\section{Howe vectors}

In this section, we will introduce and study Howe vectors. Following
\cite{Baruch:2005}, for a positive integer $m$, let
$K_n^m=I_n+M_n(\mathfrak{p}^m)$, here $\mathfrak{p}$ is the maximal
ideal of $\mathcal{O}$ and $\CO$ is the ring of integers of $F$. Use
$\varpi$ to denote an uniformizer of $F$. Let

$$d=\begin{pmatrix} 1 & \\ & \varpi^2 \\ & & \varpi^4 \\
& & & \ddots \\ & & & & \varpi^{2n-2}   \end{pmatrix}.$$

Put $J_{n,m}=d^mK_n^md^{-m}$, $N_{n,m}=N_n\cap J_{n,m}$,
$\bar{N}_{n,m}=\bar{N}_n\cap J_{n,m}$, $\bar{B}_{n,m}=\bar{B}_n\cap
J_{n,m}$ and $A_{n,m}=A_n\cap J_m$, then

$$J_{n,m}=\bar{N}_{n,m}A_{n,m}N_{n,m}= \bar{B}_{n,m}N_{n,m}=N_{n,m}\bar B_{n,m}.$$

For $j\in J_{n,m}$, write $j=\bar{b}_jn_j$ with respect to the above
decomposition, as in \cite{Baruch:2005}, define a character $\psi_m$
on $J_{n,m}$ by

$$\psi_m(j)=\psi(n_j).$$

\begin{xrem}
We will write $J_m$ for $J_{n,m}$, and $\psi$ for $\psi_m$ when
there is no confusion.
\end{xrem}

In this section we assume $\pi$ is a generic irreducible
unitarizable representation of $G_n$.

\begin{defin}
\label{def3} $W\in \mathcal{W}(\pi,\psi)$ is called a \emph{Howe
vector} of $\pi$ of level $m$ with respect to $\psi_m$ if

\begin{equation}
W(gj)=\psi_m(j)W(g) \hspace{1cm} \label{9}
\end{equation}
for all $g\in G_n$, $j\in J_{m}=J_{n,m}$.
\end{defin}

\begin{xrem}
It follows that the level $m$ must be greater than or equal to the
conductor of the central character $\omega_{\pi}$ of $\pi$ if the
Howe vector exists.
\end{xrem}

For each $W\in \mathcal{W}(\pi,\psi)$, let $M$ be a positive
constant such that $R(K_{n}^M)W=W$ where $R$ denotes the action of
right multiplication. For any $m>3M$, put

\[
W_m(g)=\int_{N_{n,m}}W(gu)\psi^{-1}(u)du,
\]
then by Lemma 7.1 in \cite{Baruch:2005}, we have

\[
W_m(gj)=\psi_m(j)W_m(g), \forall j\in J_m, \ \ \forall g\in G_n.
\]
This gives the existence of Howe vectors when $m$ is large enough.
The following lemma establishes its uniqueness in Kirillov model,
see Theorem 5.2 in \cite{cjs} for the proof.

\begin{lem}
\label{le3.2} Assume $W\in \mathcal{W}(\pi,\psi)$ satisfying (3.1).
Let $h\in G_{n-1}$, if

$$W\begin{pmatrix} h & \\ & 1 \end{pmatrix}\neq 0,$$
then $h\in N_{n-1}\bar{B}_{n-1,m}$. Moreover

$$W\begin{pmatrix} h & \\ & 1 \end{pmatrix}=\psi(u)W(I)$$
if $h=u\bar{b}$, with $u\in N_{n-1}$, $\bar{b}\in \bar{B}_{n-1,m}$.

\end{lem}

\begin{xrem}
Thus, given $\pi$, if $m$ is large enough, there exists a unique
vector $W_{v_m}\in \CW(\pi,\psi)$ satisfying (3.1) and
$W_{v_m}(I)=1$. We will call this vector as the normalized Howe
vector of level $m$ with respect to $\psi_m$.
\end{xrem}

\begin{xrem}
By the constructions above, Howe vectors exist if their levels $m$
are sufficiently large. So when we talk about Howe vectors, we
implicitly mean the levels are large enough so that these vectors
exist.
\end{xrem}

\begin{proposition}
\label{prop3.3} If $W_{v_m}$ is the normalized Howe vector of $\pi$
of level $m$ with respect to $\psi_m$, then
$\widetilde{W}_{\omega_n.v_m}$ is the normalized Howe vector of
level $m$ with respect to $\psi_m^{-1}$ for
$(\pi^*,\widetilde{\CW})$, and
\[
\overline{W}_{v_m}(g)=W_{v_m}(\omega_n {^t}g^{-1}\omega_n) \ \ \ \ \
\ \forall g\in G_n.
\]
\end{proposition}

\begin{proof}
Consider the Whittaker function $\widetilde{W}_{\omega_n.v_m}$. For
$j\in J_m$, one can check $\omega_n{^tj^{-1}}\omega_n\in J_m$, and
\[
\psi_m(\omega_n {^tj^{-1}} \omega_n)=\psi_m^{-1}(j).
\]
Thus for any $g\in G_n, j\in J_m$,
\[
\widetilde{W}_{\omega_n.v_m}(gj)=\widetilde{W}_{{^tj^{-1}}\omega_n.v_m}(g)
=\widetilde{W}_{\psi_m^{-1}(j)\omega_n.v_m}(g)=\psi_m^{-1}(j)\widetilde{W}_{\omega_n.v_m}(g).
\]
This shows that $\widetilde{W}_{\omega_n.v_m}$ is the normalized
Howe vector of level $m$ with respect to $\psi_m^{-1}$ for
$(\pi^*,\widetilde{\CW})$.

On the other hand, by Proposition \ref{prop2}, $\overline{\CW}$ is a
Whittaker model of $\pi^*$ with respect to $\psi^{-1}$. Take
$\overline{W}_{v_m}\in \overline{\CW}$, for any $g\in G_n, j\in
J_m$,
\[
\overline{W}_{v_m}(gj)=\overline{W}_{v_m}(g)\overline{\psi_m(j)}
=\psi_m^{-1}(j)\overline{W}_{v_m}(g),
\]
which implies that $\overline{W}_{v_m}$ is the normalized Howe
vector of level $m$ with respect to $\psi_m^{-1}$. Thus by the
uniqueness of Howe vectors, we have
\[
\widetilde{W}_{\omega_n.v_m}(g)=\overline{W}_{v_m}(g)
 \ \ \ \ \forall g\in G_n \hspace{0.3cm} \cdots\cdots(2),
\]
which proves the proposition.
\end{proof}

\begin{proposition}
\label{prop3.4} Assume $n=2r+1$.  Let $g=\begin{pmatrix} 0&a
\\ b\omega_{r} &0 \end{pmatrix}\begin{pmatrix} u_{r}& 0
&u_{r}\omega_{r}a'u \\ 0&1&0 \\ 0&0& I_{r}  \end{pmatrix}$, with
$a=diag(a_1,...,a_{r+1})\in A_{r+1}$, $a'=diag(a'_1,...,a'_{r})\in
A_r$, $b=diag(b_1,...,b_{r})\in A_{r}$, $u_r\in N_r, u=(u_{ij})\in
N_r$,. $W_{v_m}$ is the normalized Howe vector. If $W_{v_m}(g)\neq
0$, then $\displaystyle\frac{a_i}{a_{i+1}}\in 1+\mathfrak{p}^m$ for
$i=1,2,...,r$, $\begin{pmatrix} I_r &0& \omega_ra'u \\ 0&1&0 \\ 0&0&
I_r \end{pmatrix} \in J_m$, and
\[
W_{v_m}(g)=W_{v_m}\left( \begin{pmatrix} 0&a \\ b\omega_{r} &0
\end{pmatrix}
\begin{pmatrix} u_{r}& 0 &0 \\ 0&1&0 \\ 0&0& I_{r}  \end{pmatrix} \right).
\]
\end{proposition}

\begin{proof}  Take $j_1=\begin{pmatrix} I_r & 0 &0 \\ 0&1& x_1 \\ 0&0&I_r \end{pmatrix}\in J_m$
 with $x_1=(x_{11},...,x_{1r})$.
As $\pi(j_1).v_{m}=\psi_m(j_1)v_m$, we find

\begin{eqnarray*}
&&\psi(x_{11})W_{v_m}(g)=W_{v_m}(gj_1) \\
&=&W_{v_m}\left( \begin{pmatrix} 0&a \\ b\omega_{r} &0
\end{pmatrix}\begin{pmatrix} u_{r}& 0 &u_{r}\omega_{r}a'u \\ 0&1&0
\\ 0&0& I_{r}  \end{pmatrix}j_1 \right) \\
&=&W_{v_m}\left( \begin{pmatrix} 0&a \\ b\omega_{r} &0
\end{pmatrix}j_1 \begin{pmatrix} u_{r}& 0 &u_{r}\omega_{r}a'u \\
0&1&0 \\ 0&0& I_{r}  \end{pmatrix}  \right)\\
&=&W_{v_m}\left(\begin{pmatrix} 1& x'_1&0 \\ 0&I_r&0 \\ 0&0&I_r
\end{pmatrix}  \begin{pmatrix} 0&a \\ b\omega_{r} &0
\end{pmatrix}\begin{pmatrix} u_{r}& 0 &u_{r}\omega_{r}a'u \\ 0&1&0
\\ 0&0& I_{r}  \end{pmatrix} \right) \\
&=&\psi(a_1a_2^{-1}x_{11})W_{v_m}(g),
\end{eqnarray*}
where $x'_1=(a_1a_2^{-1}x_{11},...,a_1a_{r+1}^{-1}x_{1r} )$.

Since $W_{v_m}(g)\neq 0$, we have $\psi(x_{11}(1-a_1a_2^{-1}))=1$.
As $x_{11}\in \mathfrak{p}^{-m}$ is arbitrary, we get
$1-a_1a_2^{-1}\in \mathfrak{p}^m$, which means $a_1a_2^{-1}\in
1+\mathfrak{p}^m$.

Let $j'_1=\begin{pmatrix} I_r &0&0 \\ y_1&1&0 \\ 0&0&I_r
\end{pmatrix} \in J_m$ with $y_1=(y_{11},...,y_{1r})$. Then we have

\begin{eqnarray*}
&& W_{v_m}(g)=W_{v_m}(gj'_1) \\
&=& W_{v_m}\left( \begin{pmatrix}
0&a
\\ b\omega_{r} &0
\end{pmatrix}\begin{pmatrix} u_{r}& 0 &u_{r}\omega_{r}a'u \\ 0&1&0
\\ 0&0& I_{r}  \end{pmatrix}j'_1 \right) \\
&=& W_{v_m}\left( \begin{pmatrix} 0&a \\ b\omega_{r} &0
\end{pmatrix} \begin{pmatrix} I_r&0&0 \\  y_1u_r^{-1}&1&
-y_1\omega_ra'u \\ 0&0&I_r   \end{pmatrix} \begin{pmatrix} u_{r}& 0
&u_{r}\omega_{r}a'u \\ 0&1&0 \\ 0&0& I_{r}  \end{pmatrix} \right) \\
&=& W_{v_m} \left( \begin{pmatrix} 1&y'_1&* \\ 0&I_r&0 \\ 0&0&I_r
\end{pmatrix}  \begin{pmatrix} 0&a \\ b\omega_{r} &0
\end{pmatrix}\begin{pmatrix} u_{r}& 0 &u_{r}\omega_{r}a'u \\ 0&1&0
\\ 0&0& I_{r}  \end{pmatrix} \right) \\ &=& \psi(a_1a_2^{-1}
y_{1r}a'_1)W_{v_m}(g),
\end{eqnarray*}
where $y'_1=(a_1a_2^{-1}y_{1r}a'_1,... )$.

Since $W_{v_m}(g)\neq 0$, we have $\psi(a_1a_2^{-1} y_{1r}a'_1)=1$.
As $a_1a_2^{-1}\in 1+\mathfrak{p}^m$, and $y_{1r}\in
\mathfrak{p}^{3m}$ is arbitrary, we find $a'_1\in
\mathfrak{p}^{-3m}$.

Write
\begin{eqnarray*}
g&=& \begin{pmatrix} 0&a \\ b\omega_{r} &0
\end{pmatrix}\begin{pmatrix} u_{r}& 0 &0 \\ 0&1&0 \\ 0&0& I_{r}
\end{pmatrix}\begin{pmatrix} I_{r}& 0 &\omega_{r}a'u \\ 0&1&0 \\
0&0& I_{r}  \end{pmatrix} \\ &=&  \begin{pmatrix} 0&a \\ b\omega_{r}
&0 \end{pmatrix}\begin{pmatrix} u_{r}& 0 &0 \\ 0&1&0 \\ 0&0& I_{r}
\end{pmatrix}\begin{pmatrix} I_{r}& 0 & u_1' \\ 0&I_2&0 \\ 0&0&
I_{r-1}  \end{pmatrix}\begin{pmatrix} I_{r-1}& 0 & 0 &0&0 \\ 0&1&0
&a'_1&0 \\ 0&0&1&0&0 \\ 0&0&0&1&0 \\ 0& 0&0&0&I_{r-1}
\end{pmatrix},
\end{eqnarray*}
where the matrix $(0,u'_1)=\omega_ra'u-\begin{pmatrix} 0&0 \\ a'_1&0
\end{pmatrix}$.

Note that the last matrix belongs to $J_m$, thus if we set
\[
g_1= \begin{pmatrix} 0&a \\ b\omega_{r} &0 \end{pmatrix}
\begin{pmatrix} u_{r}& 0 &0 \\ 0&1&0 \\ 0&0& I_{r}  \end{pmatrix}
\begin{pmatrix} I_{r}& 0 & u_1' \\ 0&I_2&0 \\ 0&0& I_{r-1}  \end{pmatrix},
\]
we have
\[
W_{v_m}(g)=W_{v_m}\left(  \begin{pmatrix} 0&a \\ b\omega_{r} &0
\end{pmatrix}
\begin{pmatrix} u_{r}& 0 &0 \\ 0&1&0 \\ 0&0& I_{r}  \end{pmatrix}
\begin{pmatrix} I_{r}& 0 & u_1' \\ 0&I_2&0 \\ 0&0& I_{r-1}  \end{pmatrix}
\right)=W_{v_m}(g_1).
\]

Take $j_2=\begin{pmatrix} I_{r+1} & 0 &0 \\ 0&1& x_2 \\ 0&0&I_{r-1}
\end{pmatrix}\in J_m$ with $x_2=(x_{21},...,x_{2,r-1})$. Then we
have
\begin{eqnarray*}
&&\psi(x_{21})W_{v_m}(g)=\psi(x_{21})W_{v_m}(g_1)
=\psi(x_{21})W_{v_m}\left(  \begin{pmatrix} 0&a \\ b\omega_{r} &0
\end{pmatrix}\begin{pmatrix} u_{r}& 0 &0 \\ 0&1&0 \\ 0&0& I_{r}
\end{pmatrix}\begin{pmatrix} I_{r}& 0 & u_1' \\ 0&I_2&0 \\ 0&0&
I_{r-1}  \end{pmatrix} \right) \\
&=&W_{v_m}\left(  \begin{pmatrix} 0&a \\ b\omega_{r} &0
\end{pmatrix}\begin{pmatrix} u_{r}& 0 &0 \\ 0&1&0 \\ 0&0& I_{r}
\end{pmatrix}\begin{pmatrix} I_{r}& 0 & u_1' \\ 0&I_2&0 \\ 0&0&
I_{r-1}  \end{pmatrix} j_2 \right) \\
&=&W_{v_m}\left(  \begin{pmatrix} 0&a \\ b\omega_{r} &0
\end{pmatrix}\begin{pmatrix} u_{r}& 0 &0 \\ 0&1&0 \\ 0&0& I_{r}
\end{pmatrix}j_2 \begin{pmatrix} I_{r}& 0 & u_1' \\ 0&I_2&0 \\ 0&0&
I_{r-1}  \end{pmatrix} \right)\\
&=&W_{v_m}\left(  \begin{pmatrix} 0&a \\ b\omega_{r} &0
\end{pmatrix}j_2 \begin{pmatrix} u_{r}& 0 &0 \\ 0&1&0 \\ 0&0& I_{r}
\end{pmatrix}\begin{pmatrix} I_{r}& 0 & u_1' \\ 0&I_2&0 \\ 0&0&
I_{r-1}  \end{pmatrix} \right)\\
%\end{eqnarray*}
%\begin{eqnarray*}
&=&W_{v_m}\left(\begin{pmatrix} 1& 0&0&0 \\ 0&1&x'_2&0 \\
0&0&I_{r-1}&0 \\ 0&0&0&I_r  \end{pmatrix}  \begin{pmatrix} 0&a \\
b\omega_{r} &0 \end{pmatrix}\begin{pmatrix} u_{r}& 0 &0 \\ 0&1&0 \\
0&0& I_{r}  \end{pmatrix}\begin{pmatrix} I_{r}& 0 & u_1' \\ 0&I_2&0
\\ 0&0& I_{r-1}  \end{pmatrix}  \right) \\
&=&\psi(a_2a_3^{-1}x_{21})W_{v_m}\left(  \begin{pmatrix} 0&a \\
b\omega_{r} &0 \end{pmatrix}\begin{pmatrix} u_{r}& 0 &0 \\ 0&1&0 \\
0&0& I_{r}  \end{pmatrix}\begin{pmatrix} I_{r}& 0 & u_1' \\ 0&I_2&0
\\ 0&0& I_{r-1}  \end{pmatrix} \right)\\ &=&
\psi(a_2a_3^{-1}x_{21})W_{v_m}(g_1)=\psi(a_2a_3^{-1}x_{21})W_{v_m}(g),
\end{eqnarray*}
where $x'_2=(a_2a_3^{-1}x_{21},...,a_2a_{r+1}^{-1}x_{2,r-1} )$. \\

As $W_{v_m}(g)\neq 0$ and $x_{21}\in \mathfrak{p}^{-m}$ is
arbitrary, we must have $1-a_2a_3^{-1}\in \mathfrak{p}^m$, that is,
$a_2a_3^{-1}\in 1+\mathfrak{p}^m$. \\

Now take $j'_2=\begin{pmatrix} I_{r}&0&0&0 \\ 0&1&0&0 \\ y_2&0&1&0
\\ 0&0&0&I_{r-1} \end{pmatrix}\in J_m$ with
$y_2=(y_{21},...,y_{2r})$. Similarly as above, we have

\begin{eqnarray*}
&& W_{v_m}(g)=W_{v_m}(g_1)=W_{v_m}(g_1j'_2) \\
&=& W_{v_m}\left(   \begin{pmatrix} 0&a \\ b\omega_{r} &0
\end{pmatrix}\begin{pmatrix} u_{r}& 0 &0 \\ 0&1&0 \\ 0&0& I_{r}
\end{pmatrix}\begin{pmatrix} I_{r}& 0 & u_1' \\ 0&I_2&0 \\ 0&0&
I_{r-1}  \end{pmatrix} j'_2 \right) \\
\end{eqnarray*}
\begin{eqnarray*}
&=&W_{v_m}\left(  \begin{pmatrix} 0&a \\ b\omega_{r} &0
\end{pmatrix} \begin{pmatrix} u_{r}& 0 &0 \\ 0&1&0 \\ 0&0& I_{r}
\end{pmatrix} \begin{pmatrix} I_r&0&0&0 \\ 0&1&0&0 \\
y_2&0&1&-y_2u'_1 \\ 0&0&0&I_{r-1}   \end{pmatrix}  \begin{pmatrix}
I_{r}& 0 & u_1' \\ 0&I_2&0 \\ 0&0& I_{r-1}  \end{pmatrix} \right) \\
&=& W_{v_m}\left( \begin{pmatrix} 0&a \\ b\omega_{r} &0
\end{pmatrix} \begin{pmatrix} I_r&0&0&0 \\ 0&1&0&0 \\
y_2u_r^{-1}&0&1& -y_2u'_1 \\ 0&0&0&I_{r-1}        \end{pmatrix}
\begin{pmatrix} u_{r}& 0 &0 \\ 0&1&0 \\ 0&0& I_{r}  \end{pmatrix}
\begin{pmatrix} I_{r}& 0 & u_1' \\ 0&I_2&0 \\ 0&0& I_{r-1}
\end{pmatrix} \right) \\
&=& W_{v_m}\left(\begin{pmatrix} 1&0&0&0 \\ 0&1& y'_2 &*\\
0&0&I_{r-1}&0 \\ 0&0&0&I_{r} \end{pmatrix} \begin{pmatrix} 0&a \\
b\omega_{r} &0 \end{pmatrix} \begin{pmatrix} u_{r}& 0 &0 \\ 0&1&0 \\
0&0& I_{r}  \end{pmatrix} \begin{pmatrix} I_{r}& 0 & u_1' \\ 0&I_2&0
\\ 0&0& I_{r-1}  \end{pmatrix} \right) \\ &=&
\psi(-a_2a_3^{-1}(y_{2r}a'_1u_{12}+y_{2,r-1}a'_2
))W_{v_m}(g_1)=\psi(-a_2a_3^{-1}(y_{2r}a'_1u_{12}+y_{2,r-1}a'_2
))W_{v_m}(g),
\end{eqnarray*}
where $y'_2=(-a_2a_3^{-1}(y_{2r}a'_1u_{12}+y_{2,r-1}a'_2 ),...)$.

Again, as $W_{v_m}(g)\neq 0$, $a_2a_3^{-1}\in 1+\mathfrak{p}^m$ and
$y_{2,r-1}\in \mathfrak{p}^{7m}, y_{2r}\in \mathfrak{p}^{5m}$ are
arbitrary, we get $a'_2\in \mathfrak{p}^{-7m}, a'_1u_{12}\in
\mathfrak{p}^{-5m}$, thus it follows that
\[
g_1=\begin{pmatrix} 0&a \\ b\omega_{r} &0 \end{pmatrix}
\begin{pmatrix} u_{r}& 0 &0 \\ 0&1&0 \\ 0&0& I_{r}  \end{pmatrix}
\begin{pmatrix} I_{r}& 0 & u_1' \\ 0&I_2&0 \\ 0&0& I_{r-1}  \end{pmatrix}
\]
\[
=\begin{pmatrix} 0&a \\ b\omega_{r} &0 \end{pmatrix}
\begin{pmatrix} u_{r}& 0 &0 \\ 0&1&0 \\ 0&0& I_{r}  \end{pmatrix}
\begin{pmatrix} I_{r}& 0 & u_2' \\ 0&I_3&0 \\ 0&0& I_{r-2}  \end{pmatrix}
\begin{pmatrix} I_{r-2}&0&0&0&0&0 \\ 0& 1&0& 0 &a'_2&0 \\ 0&0&1&0&a'_1u_{12}&0 \\
0&0&0&I_2&0&0 \\ 0& 0&0&0&1&0\\ 0&0&0&0&0& I_{r-2}  \end{pmatrix},
\]
where $u'_1=\begin{pmatrix} 0&u'_2  \end{pmatrix} +\begin{pmatrix}
0&0 \\ a'_2&0 \\ a'_1u_{12}&0 \end{pmatrix}$.

Note that the last matrix belongs to $J_m$. So we find
\[
W_{v_m}(g)=W_{v_m}(g_1)=W_{v_m}(g_2)
\]
if we set $g_2=\begin{pmatrix} 0&a \\ b\omega_{r} &0
\end{pmatrix}\begin{pmatrix} u_{r}& 0 &0 \\ 0&1&0 \\ 0&0& I_{r}
\end{pmatrix}\begin{pmatrix} I_{r}& 0 & u_2' \\ 0&I_3&0 \\ 0&0&
I_{r-2}  \end{pmatrix}$. \\

Now take $j_3=\begin{pmatrix} I_{r+2}&0&0 \\ 0&1&x_3 \\ 0&0&I_{n-2}
\end{pmatrix},j'_3=\begin{pmatrix}
I_{r}&&&&\\&1&&&\\&&1&&\\y_3&&&1&\\&&&&1  \end{pmatrix} ...\in J_m$,
and then argue as above inductively, eventually, we will find
$a_3a_4^{-1}, ...,a_{r}a_{r+1}^{-1}\in 1+\mathfrak{p}^m$,
$\begin{pmatrix} I_r &0& \omega_ra'u \\ 0&1&0 \\ 0&0& I_r
\end{pmatrix} \in J_m$, and
\[
W_{v_m}(g)=W_{v_m}\left( \begin{pmatrix} 0&a \\ b\omega_{r} &0
\end{pmatrix}
\begin{pmatrix} u_{r}& 0 &0 \\ 0&1&0 \\ 0&0& I_{r}  \end{pmatrix}
\right),
\]
which finishes the proof.
\end{proof}

\begin{proposition}
\label{prop3.5} Assume $n=2r+1$.  Let $g=\begin{pmatrix} 0&a
\\ b\omega_{r} &0 \end{pmatrix}\begin{pmatrix} u_{r}& 0
&u_{r}\omega_{r}a'u \\ 0&1&0 \\ 0&0& I_{r}  \end{pmatrix}$, with
$a=diag(a_1,...,a_{r+1})\in A_{r+1}$, $a'=diag(a'_1,...,a'_{r})\in
A_r$, $b=diag(b_1,...,b_{r})\in A_{r}$, $u_r\in N_r, u=(u_{ij})\in
N_r$,. $W_{v_m}$ is the normalized Howe vector. Let
$u_1=\begin{pmatrix} u_{r}& 0 &0 \\ 0&1&0 \\ 0&0& I_{r}
\end{pmatrix}, u_2=\omega_n{^t}u_1\omega_n$. If $W_{v_m}(g)\neq
0$,  then
\[
\overline{W}_{v_m}(u_2g)=W_{v_m}(g^{-1}u_2^{-1}).
\]
\end{proposition}

\begin{proof}  By Proposition \ref{prop3.4}, if $W_{v_m}(g)\neq 0$,
then $a_ia_{i+1}^{-1}\in 1+\mathfrak{p}^m, i=1,2,...,r$,
$\begin{pmatrix} I_r &0& \omega_ra'u \\ 0&1&0 \\ 0&0& I_r
\end{pmatrix} \in J_m$ and
\[
W_{v_m}(g)=W_{v_m}(g'),
\]
where $g'=\begin{pmatrix} 0&a \\ b\omega_{r} &0 \end{pmatrix}u_1$.
Note that $W_{v_m}(g^{-1}u_2^{-1})=W_{v_m}(g'^{-1}u_2^{-1})$. Hence
it suffices to prove the proposition for $g'$, that is,
\[
\overline{W}_{v_m}(u_2g')=W_{v_m}(g'^{-1}u_2^{-1}).
\]

By Proposition \ref{prop3.3}, we have
\begin{eqnarray*}
&& \overline{W}_{v_m}(u_2g')=W_{v_m}(\omega_n {^t}u_2^{-1} {^t}g'^{-1} \omega_n) \\
&=& W_{v_m}\left(  u_1^{-1}\omega_n \begin{pmatrix} 0& a^{-1} \\
b^{-1}\omega_r&0 \end{pmatrix}  {^t}u_1^{-1} \omega_n \right) \\
&=& W_{v_m} \left( u_1^{-1} \begin{pmatrix} 0& \omega_rb^{-1} \\
\omega_{r+1}a^{-1}\omega_{r+1}&0 \end{pmatrix} u_2^{-1}  \right) \\
&=& \psi(u_1^{-1})W_{v_m}\left( \begin{pmatrix} &
\omega_rb^{-1}a_{r+1}\\ I_{r+1}& \end{pmatrix} \begin{pmatrix}
\omega_{r+1}a^{-1}\omega_{r+1}& \\ & a_{r+1}^{-1}I_r \end{pmatrix}
\begin{pmatrix} I_{r+1}& \\ & \omega_r{^t}u_r^{-1}\omega_r
\end{pmatrix}  \right) \\
&=&  \psi(u_1^{-1})W_{v_m}\left( \begin{pmatrix} &
\omega_rb^{-1}a_{r+1}\\ I_{r+1}& \end{pmatrix}  \begin{pmatrix}
I_{r+1}& \\ & \omega_r{^t}u_r^{-1}\omega_r \end{pmatrix}
\begin{pmatrix} \omega_{r+1}a^{-1}\omega_{r+1}& \\ & a_{r+1}^{-1}I_r
\end{pmatrix} \right) \\
&=& \psi(u_1^{-1})\omega_{\pi}(a_{r+1}^{-1})  W_{v_m}\left(
\begin{pmatrix} & \omega_rb^{-1}a_{r+1}\\ I_{r+1}& \end{pmatrix}
\begin{pmatrix} I_{r+1}& \\ & \omega_r{^t}u_r^{-1}\omega_r
\end{pmatrix} a'' \right), \\
\end{eqnarray*}
where $a''=diag(1,a_{r+1}a_r^{-1},...,a_{r+1}a_1^{-1},1,...,1 )$.

Since $a_ia_{i+1}^{-1}\in 1+\mathfrak{p}^m$ for $i=1,2,...,r$, we
then have $a_{r+1}a_i^{-1}\in 1+\mathfrak{p}^m$ for $i=1,...,r$. It
follows that $a''\in J_m$, and we find
\[
\overline{W}_{v_m}(u_2g')= \psi(u_1^{-1})\omega_{\pi}(a_{r+1}^{-1})
W_{v_m}\left( \begin{pmatrix} & \omega_rb^{-1}a_{r+1}\\ I_{r+1}&
\end{pmatrix}
\begin{pmatrix} I_{r+1}& \\ & \omega_r{^t}u_r^{-1}\omega_r \end{pmatrix}  \right) \ \ \ \ \ \ \cdots
(3).
\]

On the other hand,
\begin{eqnarray*}
&& W_{v_m}(g'^{-1}u_2^{-1})=W_{v_m}\left(u_1^{-1} \begin{pmatrix} 0&
\omega_rb^{-1} \\ a^{-1}&0  \end{pmatrix}u_2^{-1} \right) \\
&=& \psi(u_1^{-1})W_{v_m}\left( \begin{pmatrix} &
\omega_rb^{-1}a_{r+1}\\ I_{r+1}& \end{pmatrix} \begin{pmatrix}
a^{-1}& \\ & a_{r+1}^{-1}I_r \end{pmatrix} \begin{pmatrix} I_{r+1}&
\\ & \omega_r{^t}u_r^{-1}\omega_r \end{pmatrix}  \right) \\
&=&  \psi(u_1^{-1})W_{v_m}\left( \begin{pmatrix} &
\omega_rb^{-1}a_{r+1}\\ I_{r+1}& \end{pmatrix}  \begin{pmatrix}
I_{r+1}& \\ & \omega_r{^t}u_r^{-1}\omega_r \end{pmatrix}
\begin{pmatrix} a^{-1}& \\ & a_{r+1}^{-1}I_r \end{pmatrix} \right)
\\
&=& \psi(u_1^{-1})\omega_{\pi}(a_{r+1}^{-1})  W_{v_m}\left(
\begin{pmatrix} & \omega_rb^{-1}a_{r+1}\\ I_{r+1}& \end{pmatrix}
\begin{pmatrix} I_{r+1}& \\ & \omega_r{^t}u_r^{-1}\omega_r
\end{pmatrix} a'''\right), \\
\end{eqnarray*}
where $a'''=diag(a_{r+1}a_1^{-1},...,a_{r+1}a_r^{-1},1,1,...,1)$.

Similarly, $a'''\in J_m$, and we get
\[
W_{v_m}(g'^{-1}u_2^{-1})=\psi(u_1^{-1})\omega_{\pi}(a_{r+1}^{-1})
W_{v_m}\left( \begin{pmatrix} & \omega_rb^{-1}a_{r+1}\\ I_{r+1}&
\end{pmatrix}
\begin{pmatrix} I_{r+1}& \\ & \omega_r{^t}u_r^{-1}\omega_r \end{pmatrix} \right)  \ \ \ \ \ \ \cdots
(4).
\]

Compare (3) and (4), we find
\[
\overline{W}_{v_m}(u_2g')=W_{v_m}(g'^{-1}u_2^{-1})
\]
and the proposition follows.
\end{proof}

We will record the following analog property of $W_{v_m}$ on the big
Bruhat cell though we don't need it in the present paper.

\begin{proposition}
\label{prop3.6} For $g=u_1\omega_na u_2\in N_n\omega_nA_nN_n$, let
$u=\omega_n {^tu_2}\omega_n u_1^{-1}$, then
\[
\overline{W}_{v_m}(ug)=W_{v_m}(g^{-1}u^{-1}).
\]
\end{proposition}
\begin{proof}
Use Proposition \ref{prop3.3}.
\end{proof}

%\textbf{Remark.} The constant $M$ is the same as the one in Theorem
%7.3 in \cite{Baruch:2005}, which depends only on the representation
%$\pi$ and is uniform for all $b$. This is not explicitly stated in
%Theorem 7.1 in \cite{cjs}.

\section{Proof of the Main Result}
In this section, we will prove \textbf{Conjecture 2} which will
imply the local converse conjecture of Jacquet by the results in
\cite{JNS}. We first recall the conjecture as follows.

%\textbf{Conjecture 1.} Let $\pi_1$ and $\pi_2$ be irreducible
%generic smooth representations of $G_n$. Suppose for any integer
%$r$, with $1\leq r \leq [\frac{n}{2}]$, and any irreducible generic
%smooth representation $\rho$ of $G_r$, we have

%$$\gamma(s, \pi_1\times \rho, \psi )= \gamma(s, \pi_2\times \rho,
%\psi )$$ then $\pi_1$ and $\pi_2$ are isomorphic.

%By the work of Dihua Jiang, Chufeng Nien and Shaun Stevens in
%section 2.4, \cite{JNS}, this conjecture has been reduced to the
%following conjecture.

\textbf{Conjecture 2.} Assume $n\ge 2$. Let $\pi_1$ and $\pi_2$ be
irreducible unitarizable and supercuspidal smooth representations of
$GL_n(F)$. Suppose for any integer $r$, with $1\leq r \leq
[\frac{n}{2}]$, and any irreducible generic smooth representation
$\rho$ of $G_r$, we have

$$\gamma(s, \pi_1\times \rho, \psi )= \gamma(s, \pi_2\times \rho,
\psi ),$$ then $\pi_1$ and $\pi_2$ are isomorphic.

%We will prove this conjecture in this section which will imply
%\textbf{Conjecture 1}. Assume $\pi_1,\pi_2$ are unitarizable and
%supercuspidal representations.

From section 3.1, \cite{Ch}, we have the following disjoint
decomposition
\[
G_n=\sqcup_{i=0}^{n-1} N_n\alpha^i P_n,
\]
where $\alpha=\begin{pmatrix} &I_{n-1}\\ 1&  \end{pmatrix}$. We note
that for $1\le i\le n-1$,
\[
^t(\alpha^i)^{-1}=\alpha^i.
\]

The following is Proposition 3.1 in \cite{Ch}.

\begin{proposition}
\label{prop4.1} Let $\pi_1,\pi_2$ be two generic irreducible
representations of $G_n$ with the same central character, and let
$W_1,W_2$ be two Whittaker functions for $\pi_1,\pi_2$ respectively,
which agree on $P_n$. If the local gamma factors
$\gamma(s,\pi_1\times \rho,\psi)=\gamma(s,\pi_2\times \rho,\psi)$
for all irreducible generic smooth representation $\rho$ of $G_i$,
then $W_1,W_2$ agree on $N_n\alpha^i P_n$,
\end{proposition}

\begin{thm}
\label{thm4.2} Assume $n=2r+1$. Let $\pi_1$ and $\pi_2$ be generic
irreducible unitarizable representations of $G_n$. Suppose for any
integer $l$, with $1\leq l \leq [\frac{n}{2}]=r$, and any
irreducible generic smooth representation $\rho$ of $G_l$, we have
\[
\gamma(s, \pi_1\times \rho, \psi )= \gamma(s, \pi_2\times \rho, \psi
).
\]
Let $W^i_{v_m}$ be normalized Howe vectors of $\pi_i, i=1,2$. Then
for any $a\in A_n$, we have
\[
W^1_{v_m}(a\omega_n)=W^2_{v_m}(a\omega_n).
\]
\end{thm}

\begin{proof}

By Proposition \ref{prop3.3} , $\widetilde{W}^i_{\omega_n.v_m}$ is
the normalized Howe vector of ${\pi}^*_i$, $i=1,2$.  For any $a\in
A_n$, consider the following Rankin-Selberg integrals
\[
\gamma(s,\pi^*_i\times
\rho,\psi^{-1})\omega_{\rho}(-1)^{2r}\int_{N_r\backslash G_r}
\int_{M_{r\times r}} \widetilde{W}^i_{\omega_n.v_m} \left(
\begin{pmatrix} g&&   \\ x&I_{r}& \\ &&1   \end{pmatrix}
\begin{pmatrix} \omega_{2r} & \\ & 1 \end{pmatrix} \alpha^{r+1}a
\right)W'(g)\cdot
\]
\[
|det(g)|^{s-\frac{r+1}{2}}dxdg=\int
\widetilde{\widetilde{W^i}}_{\omega_n.v_m}\left(
\begin{pmatrix} g& \\ &I_{r+1} \end{pmatrix}\omega_{n,r} \begin{pmatrix} \omega_{2r} & \\ & 1 \end{pmatrix}
\alpha^{r+1}a^{-1}
\right)\widetilde{W'}(g)|det(g)|^{1-s-\frac{r+1}{2}}dg \cdots (5),
\]
where $\rho$ is any generic irreducible smooth representation of
$G_r$.

We first look at left hand side of (5), it equals (we will write
$\gamma(s,\pi^*_i\times \rho,\psi^{-1})$ simply as $\gamma$ to save
space)

\[
\gamma\int_{N_r\backslash G_r} \int_{M_{r\times r}}
\widetilde{W}^i_{\omega_n.v_m}\left(
\begin{pmatrix} g&&   \\ x&I_{r}& \\ &&1   \end{pmatrix}
\begin{pmatrix} \omega_{2r} & \\ & 1 \end{pmatrix} \alpha^{r+1}a \right)W'(g)|det(g)|^{s-\frac{r+1}{2}}dxdg
\]
\[
=\gamma\int_{N_r\backslash G_r} \int_{M_{r\times r}}
{W}^i_{v_m}\left( \omega_n
\begin{pmatrix}  {^t}g^{-1}&  -{^t}g^{-1}{^t}x&   \\  &I_{r}& \\ &&1   \end{pmatrix}
\begin{pmatrix} \omega_{r} & \\ & \omega_{r+1} \end{pmatrix} a^{-1} \omega_n \right)
W'(g)|det(g)|^{s-\frac{r+1}{2}}dxdg
\]
\[
=\gamma\int_{N_r\backslash G_r} \int_{G_r} {W}^i_{v_m}\left(
\omega_n \begin{pmatrix} {^t}g^{-1}& -{^t}g^{-1}{^t}\overline{x}&   \\
&I_{r}&
\\ &&1   \end{pmatrix}
\begin{pmatrix} \omega_{r} & \\ & \omega_{r+1} \end{pmatrix} a^{-1}
\omega_n
\right)W'(g)|det(g)|^{s-\frac{r+1}{2}}|det(\overline{x})|^rd\tilde{x}dg,
\]
where $d\tilde{x}$ denotes the Haar measure on $G_r$. We continue to
get the above equal to
\[
=\gamma\int_{N_r\backslash G_r} \int_{G_r} {W}^i_{v_m} \left(
\omega_n \begin{pmatrix}  {^t}g^{-1}& -{^t}\tilde{x} &   \\
&I_{r}& \\ &&1
\end{pmatrix}
\begin{pmatrix} \omega_{r} & \\ & \omega_{r+1} \end{pmatrix} a^{-1} \omega_n \right)
W'(g)|det(g)|^{s+\frac{r-1}{2}}|det(\tilde{x})|^rd\tilde{x}dg
\]
\[
=\gamma\int_{N_r\backslash G_r} \int_{G_r} {W}^i_{v_m}
\begin{pmatrix} 0&a^{-1}_{r+1}&0 \\a''\omega_r&0&0 \\
-\omega_r{^t}\tilde{x}\omega_ra''\omega_r & 0&
\omega_r{^t}g^{-1}\omega_ra'\omega_r  \end{pmatrix}
W'(g)|det(g)|^{s+\frac{r-1}{2}}|det(\tilde{x})|^rd\tilde{x}dg
\]
if we write $a^{-1}=\begin{pmatrix} a'&& \\ &a^{-1}_{r+1}& \\ &&a''
\end{pmatrix}$ with $a'=diag(a^{-1}_1,...,a^{-1}_r)$,
$a''=diag(a^{-1}_{r+2},...,a^{-1}_{2r+1})\in A_r$. \\

Write ${^t}x=-\omega_rv_r\omega_r cu_r\omega_ra''^{-1}\omega_r$
uniquely, where $v_r,u_r\in N_r, c\in A_r$. Note that ${^t}x$ runs
through an open dense subset of $G_r$, as $u_r,v_r$ run through
$N_r$ and $c$ runs through $A_r$. Then the above integral equals to
\[
\gamma\int_{N_r\backslash G_r} \int_{N_r\times A_r\times N_r}
{W}^i_{v_m} \begin{pmatrix} 0&a^{-1}_{r+1}&0 \\a''\omega_r&0&0 \\
 v_r\omega_r cu_r & 0& \omega_r{^t}g^{-1}\omega_ra'\omega_r  \end{pmatrix}
\]
\[
W'(g)|det(g)|^{s+\frac{r-1}{2}}|det(-ca''^{-1})|^r
\delta(v_r,c,u_r,a'') dgdv_rd^{*}cdu_r \hspace{1.5cm} \cdots (6),
\]
where $\delta(v_r,c,u_r,a'')$ is certain Jacobian as a function of
the indicated variables.

Let $\Omega_r=N_r\omega_rA_rN_r$. Let $N_r'=\{ u_r\in N_r:
\omega_ru_r\omega_r\in \Omega_r  \}$.

\textbf{Claim}: $N_r'$ is open dense in $N_r$.

\textit{proof of the claim}: First observe that the claim is
equivalent to the following
\[
\overline{N}_r\cap \Omega_r \ \ is  \ \ open \ \ dense \ \ in  \ \
\overline{N}_r \hspace{2cm}  \cdots (7)
\]
as $\omega_rN_r'\omega_r=\overline{N}_r\cap \Omega_r$.  \\

To prove (7), in general, if $g=(g_{ij})\in G_r$, then by
Proposition 10.3.6 in \cite{Goldfeld}(Proposition 10.3.6 is over
$\mathbb{R}$, but the proof works equally well over p-adic fields),
$g\in \Omega_r$ if and only if all the bottom left minors are
nonzero, that is,
$g_{r1}\neq 0, det\begin{pmatrix}g_{r-1,1}&g_{r-1,2} \\
g_{r1}&g_{r2}
\end{pmatrix}\neq 0, det\begin{pmatrix} g_{r-2,1}&g_{r-2,2}&g_{r-2,3} \\
g_{r-1,1}&g_{r-1,2}&g_{r-1,3}\\ g_{r1}&g_{r2}&g_{r2}
\end{pmatrix}\neq 0,..., det(g)\neq 0$.

Thus the complements of $\overline{N}_r\cap \Omega_r$ in
$\overline{N}_r$ is a union of finitely many closed subvarieties
with strictly smaller dimensions than $\overline{N}_r$, thus we
obtain (7) and the claim follows.

\hspace{15cm} $\Box$

Let's continue the proof of the theorem. For any $c\in A_r$, the set
\[
N_r\omega_r cu_r \omega_r A_rN_r=\Omega_r
\]
is open dense in $G_r$ if $u_r\in N_r'$, which implies the set
\[
\omega_rN_r\omega_r cu_r \omega_r A_rN_r=\omega_r\Omega_r
\]
is also open dense in $G_r$ if $u_r\in N_r'$. Hence if $u_r\in
N_r'$, the subset of cosets
\[
\Omega_r':= \overline{N}_r\cdot cu_r\omega_rA_rN_r=\omega_r\Omega_r
\]
is open dense in $\overline{N}_r\backslash G_r$.

Since $g\in N_r\backslash G_r$ if and only if $^t{g}^{-1}\in
\overline{N}_r\backslash G_r$, and $\Omega_r'$ is open dense in the
latter space, combining with the claim, we can rewrite the integral
(6) as

\[
\gamma\int_{N_r\backslash G_r} \int_{N_r\times A_r\times N_r'}
{W}^i_{v_m} \begin{pmatrix} 0&a^{-1}_{r+1}&0 \\a''\omega_r&0&0 \\
 v_r\omega_r cu_r & 0& \omega_r{^t}g^{-1}\omega_ra'\omega_r  \end{pmatrix}
\]
\[
 W'(g)|det(g)|^{s+\frac{r-1}{2}}|det(-ca''^{-1})|^r
\delta(v_r,c,u_r,a'') dgdv_rd^{*}cdu_r
\]
\[
=\gamma\int_{\Omega'_r} \int_{N_r\times A_r\times N_r'}
{W}^i_{v_m} \begin{pmatrix} 0&a^{-1}_{r+1}&0 \\a''\omega_r&0&0 \\
v_r\omega_r cu_r & 0& -v_r \omega_r cu_r \omega_r a''^{-1}v'_rb
\end{pmatrix}
\]
\[
 W'(g)|det(g)|^{s+\frac{r-1}{2}}|det(-ca''^{-1})|^r
\delta(v_r,c,u_r,a'') dgdv_rd^{*}cdu_r  \hspace{1.5cm} \cdots (8)
\]
if we write ${^t}g^{-1}=-\omega_rv_r \omega_r cu_r \omega_r
a''^{-1}v'_rb\omega_ra'^{-1}\omega_r$ where $b\in A_r, v_r'\in N_r$.
\\

Now we are going to show that the integrals in (8) are equal to each
other for $i=1,2$ based on results in section 3 and Proposition
\ref{prop4.1}, which will imply the left sides of (5) are also
equal.

Let
\[
h=\begin{pmatrix} 0&a^{-1}_{r+1}&0 \\a''\omega_r&0&0 \\
-\omega_r{^t}x\omega_ra''\omega_r & 0&
\omega_r{^t}g^{-1}\omega_ra'\omega_r  \end{pmatrix}.
\]
Direct matrix computation verifies
\[
h=\begin{pmatrix} 0&a^{-1}_{r+1}&0 \\a''\omega_r&0&0 \\
-\omega_r{^t}x\omega_ra''\omega_r & 0&
\omega_r{^t}g^{-1}\omega_ra'\omega_r  \end{pmatrix}
=\begin{pmatrix} 0&a^{-1}_{r+1}&0 \\a''\omega_r&0&0 \\
v_r\omega_r cu_r & 0& -v_r \omega_r cu_r \omega_r a''^{-1}v'_rb
\end{pmatrix}
\]
\[
=\begin{pmatrix} 1&0&0 \\ 0&v'_r& a''\omega_ru_r^{-1}c^{-1}\omega_r
\\ 0&0&v_r \end{pmatrix}
\begin{pmatrix} 0&a^{-1}_{r+1}&0 \\ 0&0&bI_r \\ \omega_rc&0&0  \end{pmatrix}
\begin{pmatrix} u_r&0& -u_r\omega_ra''^{-1}v'_rb \\ 0&1&0 \\ 0&0&I_r
\end{pmatrix}.
\]

Then apply Proposition \ref{prop3.5}, there exists some $u\in
N_{2r+1}$, independent of $\pi_i$, i=1,2, such that
\[
\overline{W}^i_{v_m}(uh)=W^i_{v_m}(h^{-1}u^{-1})
\]
whenever they are nonzero. \\

On the other hand, we can rewrite the above $h$ as $h=\omega_n
{^t}p^{-1}d \alpha^{r+1} \omega_n$ for some element $p\in P_n$, the
mirabolic subgroup, and $d$ in the center of $G_n$ with
\[
{^t}p^{-1}d=
\begin{pmatrix} {^t}g^{-1}& -{^t}x &   \\ &I_{r}& \\ &&1   \end{pmatrix}
\begin{pmatrix} \omega_{2r} & \\ & 1 \end{pmatrix} \alpha^{r+1}a^{-1}\alpha^r.
\]
Thus
\[
\overline{W}^i_{v_m}(uh)=W^i_{v_m}(h^{-1}u^{-1})=W^i_{v_m}(\omega_n
\alpha^r d^{-1} {^t}p\omega_n u^{-1}) =\omega_{\pi_i}^{-1}(d)
\widetilde{W}^i_{\omega_n.v_m}(\alpha^rp^{-1}
\omega_n{^t}u\omega_n),
\]
where $\omega_{\pi_i}$ is the central character of $\pi_i$.  Note
that the element $\alpha^rp^{-1} \omega_n{^t}u\omega_n$ belongs to
the double coset $N_{2r+1}\alpha^{r} P_n$, and it is here we need to
require the number of twists is at least $[\frac{n}{2}]=r$ as we
will see.

By Corollary 2.7, \cite{JNS}, $\pi_1,\pi_2$ have the same central
characters. By Lemma 3.2, $\widetilde{W}^i_{\omega_n.v_m}$ agree on
$P_n$, $i=1,2$. As $\pi_1,\pi_2$ have the same local gamma factors
twisted by irreducible generic representations of $G_l$ with $1\le
l\le [\frac{n}{2}]=r$, by Proposition \ref{prop4.1},
$\widetilde{W}^i_{\omega_n.v_m}, i=1,2$ also agree on
$N_n\alpha^rP_n$. As $\alpha^rp^{-1} \omega_n{^t}u\omega_n$ is an
element in $N_n\alpha^rP_n$, by the above computation we find
\[
\overline{W}^1_{v_m}(uh)=\omega_{\pi_1}^{-1}(d)
\widetilde{W}^1_{\omega_n.v_m}(\alpha^rp^{-1}
\omega_n{^t}u\omega_n)=\omega_{\pi_2}^{-1}(d)
\widetilde{W}^2_{\omega_n.v_m}(\alpha^rp^{-1}
\omega_n{^t}u\omega_n)= \overline{W}^2_{v_m}(uh)
\]
whenever they are nonzero, which implies
$W^1_{v_m}(h)=W^2_{v_m}(h)$. Hence the integrals in (8) are equal to
each other for $i=1,2$. It is clear that we need to require the
number of twists is at least $[\frac{n}{2}]=r$ as $\alpha^rp^{-1}
\omega_n{^t}u\omega_n\in N_n\alpha^rP_n$.

Since $\gamma(s, \pi^*_i\times \rho,\psi)\gamma(1-s, \pi_i\times
\rho^*, \psi^{-1})=1 $ by the statements after Lemma 3.1 in
\cite{Hen}, by the assumptions on local gamma factors, we get
$\gamma(s, \pi^*_1\times \rho,\psi)=\gamma(s, \pi^*_2\times
\rho,\psi)$. Then we can conclude that the left hand sides of (5)
are equal for $i=1,2$, which means the right hand sides are also
equal to each other.

%Since we are assuming $\pi_i, i=1,2$ are supercuspidal, their
%Whittaker functions have compact support module $Z_nN_n$. On the
%right hand side of (5), the Whittaker functions we have are of the
%form
%\[
%\widetilde{\widetilde{W^i} }_{v'}\begin{pmatrix} g_r& \\ &I_{r+1}
%\end{pmatrix}
%\]
%with $g_r\in G_r$. Thus these Whittaker functions have compact
%supports module $N_r$.
Now apply Lemma \ref{le2.3}, we conclude that
\[
\widetilde{\widetilde{W^1}}_{\omega_n.v_m}\left( \begin{pmatrix} g&
\\ &I_{r+1} \end{pmatrix} \omega_{n,r} \begin{pmatrix} \omega_{2r} &
\\ & 1 \end{pmatrix} \alpha^{r+1}a^{-1} \right)=
\widetilde{\widetilde{W^2}}_{\omega_n.v_m}\left( \begin{pmatrix} g&
\\ &I_{r+1} \end{pmatrix}\omega_{n,r}
\begin{pmatrix} \omega_{2r} & \\ & 1 \end{pmatrix} \alpha^{r+1}a^{-1}
\right) \cdots (9).
\]

Let $g=\omega_r$ in the above identity, we finally proved the
theorem.

\end{proof}

\begin{thm}
\label{thm4.3} \textbf{Conjecture 2} is true when $n=2r+1$ is odd.
\end{thm}

\begin{proof}
Let $\pi_1,\pi_2$ be irreducible unitarizable supercuspidal
representations of $G_n$ satisfying the assumptions in
\textbf{Conjecture 2}. By Theorem \ref{thm4.2}, their normalized
Howe vectors $W^{i}_{v_m}, i=1,2$ satisfy
\[
W^1_{v_m}(a\omega_n)=W^2_{v_m}(a\omega_n).
\]
As this is true for all levels of Howe vectors, then by Proposition
5.3 in \cite{cjs}, we find
\[
j_{\pi_1}(a\omega_n )=j_{\pi_2}(a\omega_n)
\]
for all $a\in A_n$. Thus $j_{\pi_1}(g)=j_{\pi_2}(g)$ for all $g\in
\Omega_n$.  \\

Consider $g=u_1\omega_nau_2\in \Omega_n$, by the weak kernel formula
Theorem \ref{thm1}, we have
\[
W^i_{v_m}(g)=W^i_{v_m}(u_1\omega_nau_2)=\psi(u_1)W^i_{u_2.v_m}(\omega_na)
\]
\[
=\int j_{\pi_i}\left(b\omega_n \begin{pmatrix} a_1 & \\ x_{21} & a_2 \\ & & \ddots \\
x_{n-1,1} & \cdots & x_{n-1,n-2} & a_{n-1} \\ & & & &  1
\end{pmatrix}^{-1} \right) W^i_{u_2.v_m}\begin{pmatrix} a_1 & \\ x_{21} & a_2 \\
& & \ddots \\ x_{n-1,1} & \cdots & x_{n-1,n-2} & a_{n-1}\\ & & & & 1
\end{pmatrix}
\]
\[
|a_1|^{-(n-1)}da_1|a_2|^{-(n-2)}dx_{21}da_2\cdots
|a_{n-1}|^{-1}dx_{n-1,1}\cdots dx_{n-1,n-2}da_{n-1}
\]
\[
=\int j_{\pi_i}\left(b\omega_n \begin{pmatrix} a_1 & \\ x_{21} & a_2 \\ & & \ddots \\
x_{n-1,1} & \cdots & x_{n-1,n-2} & a_{n-1} \\ & & & &  1
\end{pmatrix}^{-1} \right) W^i_{v_m}\left(\begin{pmatrix} a_1 & \\ x_{21} & a_2 \\
& & \ddots \\ x_{n-1,1} & \cdots & x_{n-1,n-2} & a_{n-1}\\ & & & & 1
\end{pmatrix}u_2\right)
\]
\[
|a_1|^{-(n-1)}da_1|a_2|^{-(n-2)}dx_{21}da_2\cdots
|a_{n-1}|^{-1}dx_{n-1,1}\cdots dx_{n-1,n-2}da_{n-1}.
\]

We note that the element
\[
\left(b\omega_n \begin{pmatrix} a_1 & \\ x_{21} & a_2 \\ & & \ddots \\
x_{n-1,1} & \cdots & x_{n-1,n-2} & a_{n-1} \\ & & & &  1
\end{pmatrix}^{-1} \right)
\]
is in $\Omega_n$, so
\[
j_{\pi_1}\left(b\omega_n \begin{pmatrix} a_1 & \\ x_{21} & a_2 \\ & & \ddots \\
x_{n-1,1} & \cdots & x_{n-1,n-2} & a_{n-1} \\ & & & &  1
\end{pmatrix}^{-1} \right)=j_{\pi_2}\left(b\omega_n \begin{pmatrix} a_1 & \\ x_{21} & a_2 \\ & & \ddots \\
x_{n-1,1} & \cdots & x_{n-1,n-2} & a_{n-1} \\ & & & &  1
\end{pmatrix}^{-1} \right).
\]

The element
\[
\left(\begin{pmatrix} a_1 & \\ x_{21} & a_2 \\
& & \ddots \\ x_{n-1,1} & \cdots & x_{n-1,n-2} & a_{n-1}\\ & & & & 1
\end{pmatrix}u_2\right)
\]
is in $P_n$. By Lemma \ref{le3.2}, $W^1_{v_m}(p)=W^2_{v_m}(p),
\forall p\in P_n$. Hence the last integrals are equal to each other
for $i=1,2$, which implies $W^1_{v_m}(g)=W^1_{v_m}(g)$ for all $g\in
\Omega_n$. As $\Omega_n$ is open dense in $G_n$, we eventually get
that $W^1_{v_m}(g)=W^2_{v_m}(g)$ for all $g\in G_n$ which finishes
the proof by the multiplicity one theorem on Whittaker models.
\end{proof}

\begin{thm}
\label{thm4.4} \textbf{Conjecture 2} is true when $n=2r$ is even.
\end{thm}
\begin{proof}
Suppose $\pi_1,\pi_2$ are irreducible unitarizable supercuspidal
representations of $G_{2r}$ satisfying the assumptions in
\textbf{Conjecture 2}. Take a unitary character
 $\chi$
 of $G_{1}$, and form the normalized induced representations $\tau_1=Ind(\pi_1\otimes \chi)$,
 $\tau_2=Ind(\pi_2\otimes \chi)$. By Theorem 4.2 in \cite{BZ}, both $\tau_1, \tau_2$ are irreducible generic
 smooth unitarizable representations of $G_{2r+1}$. For
any $l$ with $1\le l\le r$,
 and any irreducible generic smooth representation $\rho$ of $G_l$, we have
\[
\gamma(s, \pi_1\times \rho, \psi)=\gamma(s, \pi_2\times \rho, \psi).
\]
By the multiplicativity of local gamma factors, we get
\[
\gamma(s, \tau_1\times \rho, \psi)=\gamma(s, \tau_2\times \rho,
\psi).
\]
Then by \textbf{Theorem 4.2}, for all normalized Howe vector
$W^i_{v_m}$ of $\tau_i, i=1,2$, we have
\[
W^1_{v_m}(a\omega_n)=W^2_{v_m}(a\omega_n)\ \ \ \ \cdots (10).
\]
In the following, we will present three different approaches to
finish the proof. The first is based on well expected property of
Bessel functions: local integrability. The other two approaches are
based on well established results: Derivatives of smooth
representations of $G_n$ and Shahidi's formula expressing local
coefficients as Mellin transforms of partial Bessel functions. All
three approaches have its own interests and they are quite
independent to one another. They all illustrate the power of Bessel
functions.   \\

\textit{The first approach}. The first way is based on the well
expected property: local integrability of Bessel functions. As (10)
is true for all levels of Howe vectors, by Proposition 5.3 in
\cite{cjs}, we find
\[
j_{\pi_1}(a\omega_n )=j_{\pi_2}(a\omega_n)
\]
for all $a\in A_n$. Thus $j_{\pi_1}(g)=j_{\pi_2}(g)$ for all $g\in
\Omega_n$. It then follows from Corollary 7.2 \cite{cjs16} that
$\pi_1\cong \pi_2$. This finishes the first approach.   \\

\textit{The second approach}. The second is based on the theory of
derivatives of smooth representations on $G_n$. We first recall a
result of Cogdell and Piatetski-Shapiro about derivatives. Let $\pi$
be an irreducible generic representations of $G_n$. Take a Whittaker
function $W\in \CW(\pi,\psi)$, and a Schwartz function $\Phi_0\in
\CS(F^{n-1})$ which is supported in a sufficiently small
neighborhood of $0$, if the first derivative $\pi^{(1)}$ of $\pi$ is
irreducible, then there is a Whittaker function $W'\in
\CW(\pi^{(1)},\psi)$, such that for all $g\in G_{n-1}$, we have
\[
W\begin{pmatrix} g& \\ &1
\end{pmatrix}\Phi_0(\epsilon_{n-1}g)=|det(g)|^{1/2}W'(g)\Phi_0(\epsilon_{n-1}g),
\]
where $\epsilon_{n-1}=(0,...,0,1)$. This is a special case of the
second half of corollary to Proposition 1.7 in \cite{CPS2}.    \\

Let $W^i_{v_m}$ be the normalized Howe vector of level $m$ in
$\tau_i, i=1,2$.
Recall $\alpha=\begin{pmatrix} & I_{n-1} \\
1&
\end{pmatrix}$. Take $W=W^i_{\alpha.v_m}$, $\Phi_0$ to be the characteristic function of a sufficiently
small neighborhood of $0$. Note that $\tau_i^{(1)}\cong \pi_i$ by
Lemma 4.5 in \cite{BZ} and it is irreducible. Apply the above result
of Cogdell and Piatetski-Shapiro, we conclude that there exists some
$W'_i\in \CW(\tau_i^{(1)},\psi)$, such that
\[
W_{\alpha.v_m}^i\begin{pmatrix} g& \\ &1
\end{pmatrix}\Phi_0(\epsilon_{n-1}g)=|det(g)|^{1/2}W'_i(g)\Phi_0(\epsilon_{n-1}g)
\cdots (11).
\]

Given $g\in G_{n-1}, j\in J_{n-1,m}$, where $J_{n-1,m}$ is the same
as in the \textbf{Definition \ref{def3}}. Choose $z$ in the center
of $G_{n-1}$ so that $\Phi_0(\epsilon_{n-1}gz)=1$ and
$\Phi_0(\epsilon_{n-1}gzj)=1$. By (11), we have
\[
W_{\alpha.v_m}^i\begin{pmatrix} gzj& \\ &1
\end{pmatrix}=|det(gzj)|^{1/2}W'_i(gzj)
\cdots (12)
\]
and
\[
W_{\alpha.v_m}^i\begin{pmatrix} gz& \\ &1
\end{pmatrix}=|det(gz)|^{1/2}W'_i(gz)
\cdots (13).
\]

On the other hand, note that $\begin{pmatrix} j& \\ &1
\end{pmatrix}\alpha=\alpha \begin{pmatrix} 1& \\ &j
\end{pmatrix}$, $\begin{pmatrix} 1& \\ &j
\end{pmatrix}\in J_{n,m}$ and $\psi\begin{pmatrix} 1& \\ &j
\end{pmatrix}=\psi\begin{pmatrix} j& \\ &1
\end{pmatrix}=\psi(j)$. Then the left hand side of (12) is
\begin{eqnarray*}
&&W_{\alpha.v_m}^i\begin{pmatrix} gzj& \\ &1
\end{pmatrix} \\
&=& W_{v_m}^i\left(\begin{pmatrix} gz& \\ &1
\end{pmatrix}\alpha \begin{pmatrix} 1& \\ &j
\end{pmatrix}\right)  \\
&=&\psi(j)W_{\alpha.v_m}^i\begin{pmatrix} gz& \\ &1
\end{pmatrix} \\
%&=& \psi(j)W_{\alpha.v_m}\begin{pmatrix} gz& \\ &1
%\end{pmatrix} \\
&=& \psi(j)|det(gz)|^{1/2}W'_i(gz) \hspace{2cm} (by \ \ (13)).
\end{eqnarray*}
This equals the right hand side of (12), hence
\[
\psi(j)|det(gz)|^{1/2}W'_i(gz)=|det(gzj)|^{1/2}W'_i(gzj),
\]
which implies that
\[
\psi(j)W'_i(gz)=W'_i(gzj).
\]
As $\tau_i^{(1)}\cong \pi_i$ and it has a central character. It
follows that
\[
\psi(j)W'_i(g)=W'_i(gj),
\]
which proves that $W'_i$ is the Howe vector of level $m$ for $\pi_i,
i=1,2$.   \\

Now suppose $a\in G_{n-1}$ is diagonal, choose $z$ in the center of
$G_{n-1}$ which is sufficiently close to $0$ so that
$\Phi_0(\epsilon_{n-1}\omega_{n-1}az)=1$. Apply (11) to
$g=\omega_{n-1}az$, we have
\[
W^i_{\alpha.v_m}\begin{pmatrix} \omega_{n-1}az& \\ &1
\end{pmatrix}=
|det(\omega_{n-1}az)|^{1/2}W'_i(\omega_{n-1}az).
\]
As $W^i_{\alpha.v_m}\begin{pmatrix} \omega_{n-1}az& \\ &1
\end{pmatrix}=W^i_{v_m}\begin{pmatrix} &\omega_{n-1}az \\ 1&
\end{pmatrix}$, it follows from (10) that
\[
W'_1(\omega_{n-1}az)=W'_2(\omega_{n-1}az),
\]
which implies
\[
W'_1(\omega_{n-1}a)=W'_2(\omega_{n-1}a).
\]
As this is true for all Howe vectors $W'_i$ and all diagonal $a$, we
conclude that the Bessel functions of $\pi_1,\pi_2$ are equal to
each other by Proposition 5.3 in \cite{cjs}. Since $\pi_1,\pi_2$ are
supercuspidal representations, as in the proof of \textbf{Theorem
\ref{thm4.3}}, it follows from the weak kernel formula
(\textbf{Theorem \ref{thm1}}) that they are in fact isomorphic. This
ends the proof of the second approach.
\\

\textit{The third approach}. We first need to recall Shahidi's
formula (Theorem 6.2 in \cite{Shahidi:2002}) expressing local
coefficients as Mellin transform of partial Bessel functions in our
case. Let $P=MU$ be the standard parabolic subgroup of $G_{2r}$ with
Levi $M=G_r\times G_r$. $U$ is the unipotent part, with opposite
$\bar U$. Put $\omega_0=\begin{pmatrix} &I_r\\ I_r&
\end{pmatrix}$ and $\omega_M=\begin{pmatrix} \omega_r&\\ &\omega_r
\end{pmatrix}$. Let $N_M=N_{2r}\cap M$. Use $Z_M,
Z$ to denote the centers of $M$ and $G_{2r}$ respectively. Let
$Z_M^0=Z\backslash Z_M$.   \\

As in \cite{Shahidi:2002}, we start with the following decomposition
\[
\omega_{0}^{-1}n=mn'\bar n   \hspace{5cm} \cdots (*1)
\]
valid for almost all $n\in U$, where $m\in M, n'\in U, \bar n\in
\bar U$. The Bruhat decomposition of $m$ is
\[
m=u_1t\omega u_2  \hspace{5cm} \cdots (*2),
\]
where $u_1,u_2\in N_m, t\in A_{2r}$ and $\omega$ is certain Weyl
group element of $M$. As in section 3 of \cite{CPSSH:2008}, if we
set $u'=\omega_0u_1^{-1}\omega_0^{-1}$ and $n_1=u'n(u')^{-1}$, then
the map $n\to n_1$ gives a bijection from the set of all $n$
satisfying (*1) onto the Bruhat double coset $\bar
B_{2r}\omega_0\omega \bar UN_M$ of $G_{2r}$.  \\

Shahidi's formula involves certain unipotent integral defining
partial Bessel functions (see (*3) below). For this integral to be
nonzero, $m\in M$ appearing in the integration must support a Bessel
function in the sense of \cite{CPSSH:2005}, at least for some full
measure subset. Note that the cell $\bar B_{2r}\omega_{2r} \bar
UN_M$ is the unique Bruhat double coset of $G_{2r}$ intersecting $U$
in an open dense subset. By Proposition 3.2 in \cite{CPSSH:2008}, we
have $\omega_{2r}=\omega_0\omega$ which implies that
$\omega=\omega_M$, and this Weyl element does support a Bessel
function.    \\

$Z_M^0N_M$ acts on $U$ by conjugation, we will use
$Z_M^0N_M\backslash U$ to denote $Z_M^0N_M$ orbits in $N$, $dn$ is
certain measure on this set of orbits. We first consider $N_M$
orbits in $U$. For $\begin{pmatrix} u_1& \\ &u_2 \end{pmatrix}\in
N_M, \begin{pmatrix} I_r& X \\ &I_r \end{pmatrix}\in U$, we have
\[
\begin{pmatrix} u_1& \\ &u_2 \end{pmatrix}\begin{pmatrix} I_r& X \\ &I_r \end{pmatrix}
\begin{pmatrix} u_1& \\ &u_2 \end{pmatrix}^{-1}=\begin{pmatrix} I_r& u_1Xu_2^{-1} \\ &I_r
\end{pmatrix}.
\]
Hence the matrices like
\[
\begin{pmatrix} I_r& \omega_r t \\ &I_r \end{pmatrix}
\]
with $t\in A_r$, form a set of representatives of an open dense
subset of $N_M\backslash U$. Direct computation shows that the
decomposition (*1) for such matrices is
\[
\begin{pmatrix} I_r& \omega_rt \\ &I_r \end{pmatrix}=
\begin{pmatrix} -(\omega_rt)^{-1}&  \\ &\omega_rt \end{pmatrix}
\begin{pmatrix} I_r& -\omega_rt \\ &I_r \end{pmatrix}
\begin{pmatrix} I_r&  \\ (\omega_rt)^{-1} &I_r \end{pmatrix}.
\]
It then follows that we can find a set of representatives of a full
measure subset $\Omega$ of $Z_M^0N_M\backslash U$, and satisfy
decomposition $\omega_{0}^{-1}n=mn'\bar n$, where $n\in \Omega$ and
$m$ has the form $\omega_Ma$ for certain diagonal matrices $a\in
A_{2r}$. This is a weak version of Proposition 4.2.3 in \cite{Ts}.
\\

Let $\pi, \rho$ be generic irreducible representations of $G_r$,
denote by $\sigma=\pi\otimes \rho$ which is a generic irreducible
representations of $M$. Then the central characters
$\omega_\sigma=\omega_\pi\otimes \omega_\rho$. We also define for
$t\in F^*$, define characters of $F^*$ by
$\omega_\sigma(t)=\omega_\sigma(\alpha^\vee(t))$ and
$(\omega_0.\omega_\sigma)(t)=\omega_\sigma(\omega_0^{-1}\alpha^\vee(t)\omega_0)$,
where $\alpha^\vee (t)=\begin{pmatrix} tI_r& \\ &I_r \end{pmatrix}$.
\\

Now if $W_{\tilde v}$ is a Whittaker function in $\sigma$ with
$W_{\tilde v}(I_{2r})=1$. Let $\bar U_0$ be a sufficiently large
open compact subgroup of $\bar U$ and $\phi$ its characteristic
function. For $n\in \Omega, \omega_{0}^{-1}n=mn'\bar n, y\in F^*$,
we define the partial Bessel function
\[
j_{\tilde v,\bar U_0}(m,y)=\int_{N_{M,n}\backslash N_M} W_{\tilde
v}(mu^{-1})\phi(zu\bar nu^{-1}z^{-1})\psi(u)du \cdots (*3),
\]
where $N_{M,n}=\{u\in N_M:unu^{-1}=n \}$ and
$z=\alpha^\vee(y^{-1}.\dot{\chi}_\alpha)$ is certain element in
$Z_M^0$. If the character
$\omega_\sigma(\omega_0.\omega_\sigma^{-1})$ is ramified, Theorem
6.2 in \cite{Shahidi:2002} for the local coefficient $C(s,\sigma)$,
applied to our case, can be stated as follows.
\[
C(s,\sigma)^{-1}=\gamma(2<\tilde
\alpha,\alpha^\vee>s,\omega_\sigma(\omega_0.\omega_\sigma^{-1}),\psi
)^{-1}
\]
\[
\times \int_{Z_M^0N_M\backslash U} j_{\tilde v,\bar
U_0}(\dot{m},y_0)\omega_{(\sigma)_s}^{-1}(\dot{\chi}_\alpha)
(\omega_{0}.\omega_{(\sigma)_s})(\dot{\chi}_\alpha)q^{(s\tilde
\alpha+\rho, H_M(\dot{m}))}d\dot{n} \cdots (*4).
\]
In this formula, $y_0\in F^*$ is an element with $ord(y_0)=-d-f$,
where $d, f$ are conductors of $\psi$ and
$\omega_\sigma^{-1}(\omega_0.\omega_\sigma)$, respectively. The
choice of $y_0$ is irrelevant. The above integral is independent of
the choice of $\tilde v$ and $\bar U_0$ as long as $W_{\tilde
v}(I_{2r})=1$ and $\bar U_0$ is a sufficiently large compact open
subgroup of $\bar U$. $\dot{m}, \dot{n}$ are related by (*1).
Moreover by choosing representatives $\dot{n}$ in $\Omega$,
$\dot{m}$ have the form $\omega_Ma$ for certain diagonal matrices
$a\in A_{2r}$ as we discussed above. We refer to \cite{Shahidi:2002}
for the unexplained terms in the formula.    \\

We also note that, by Lemma 3.11 in \cite{CPSSH:2008}, the domain of
integration in the definition of $j_{\tilde v, \bar U_0}(m,y_0)$ is
independent of $m$, and depends only on $y_0$ and $\bar U_0$. \\

Now we begin the third proof. Let $\rho$ be an irreducible generic
representation of $G_{2r}$, choose a character $\chi'$ of $G_1$, so
that the normalized induced representation $\sigma=Ind(\rho\otimes
\chi')$ is generic and irreducible. Consider $\tau_i\otimes \sigma$,
which is an irreducible generic representation of $M=G_{2r+1}\times
G_{2r+1}$. The central character of $\tau_i\otimes \sigma$ is
$\omega_{\tau_i}\otimes (\omega_\rho\chi')$. Recall that for $t\in
F^*$, $\omega_\sigma(t)=\omega_\sigma(\alpha^\vee(t))$ and
$(\omega_0.\omega_\sigma)(t)=\omega_\sigma(\omega_0^{-1}\alpha^\vee(t)\omega_0)$,
where $\alpha^\vee (t)=\begin{pmatrix} tI_r& \\ &I_r \end{pmatrix}$.
Hence $\omega_{\tau_i\otimes
\sigma}(\omega_0.\omega_{\tau_i\otimes\sigma}^{-1})=\omega_{\tau_i}\cdot
(\omega_\rho\chi')^{-1}$. So we can choose $\chi'$ to require
further that the characters $\omega_{\tau_i\otimes
\sigma}(\omega_0.\omega_{\tau_i\otimes\sigma}^{-1}), i=1,2$ are
ramified.
\\

Now we want to apply Shahidi's formula (*4) to $\tau_i\otimes
\sigma, i=1,2$, and to show that $C(s, \tau_1\otimes \sigma)=C(s,
\tau_2\otimes \sigma)$. For this purpose, we first choose $\bar U_0$
large enough satisfying (*4) for both $\tau_i\otimes \sigma, i=1,2$
and fix $y_0$. Then take positive integer $l$ sufficiently large so
that $N_{2r+1,l}\times N_{2r+1,l}$ contains the domain of
integration in $j_{\tilde v_i,\bar U_0}(m, y_0)$, where
$N_{2r+1,l}=N_{2r+1}\cap J_{2r+1,l}$ as in section 3. Now choose
$W_{\tilde v_i}\begin{pmatrix} g_1&\\ &g_2
\end{pmatrix}=W^i_{v_l}(g_1)W'(g_2)$, where $W'$ is a Whittaker
function of $\sigma$ with $W'(I_{2r+1})=1$.   \\

So with this $W_{\tilde v_i}$, and plug the integral defining
$j_{\tilde v_i, \bar U_0}$ into the formula (*4) for $C(s,
\tau_i\otimes \sigma)$. We get formula
\[
C(s, \tau_i\otimes \sigma)^{-1}=\gamma(2<\tilde
\alpha,\alpha^\vee>s,\omega_{\tau_i\otimes
\sigma}(\omega_0.\omega_{\tau_i\otimes \sigma}^{-1}),\psi
)^{-1}\times
\]
\[
\int_{Z_M^0N_M\backslash U} \int_{N_{M,n}\backslash N_M} W_{\tilde
v_i}(\dot{m}u^{-1})\phi(zu\bar
nu^{-1}z^{-1})\psi(u)du\omega_{(\tau_i\otimes
\sigma)_s}^{-1}(\dot{\chi}_\alpha)
(\omega_{0}.\omega_{(\tau_i\otimes
\sigma)_s})(\dot{\chi}_\alpha)q^{(s\tilde \alpha+\rho,
H_M(\dot{m}))}d\dot{n}.
\]
Note that $\dot{m}$ has particular form
\[
\begin{pmatrix} \omega_{2r+1}a& \\ &\omega_{2r+1}b  \end{pmatrix}
\]
with diagonal matrices $a,b$, and the integration with $u$ is over
$N_{2r+1,l}\times N_{2r+1,l}$. By (10) and the definition of Howe
vectors $W^i_l$, we have
\[
W^1_l(\omega_{2r+1}au)=W^2_l(\omega_{2r+1}au)
\]
for all diagonal matrices $a$ and $u\in N_{2r+1,l}$. This then
implies that
\[
W_{\tilde v_1}(\dot{m}u^{-1})=W_{\tilde v_2}(\dot{m}u^{-1}),
\]
which means $C(s, \tau_1\otimes \sigma)=C(s, \tau_2\otimes \sigma)$.
\\

By the relation between local coefficient $C(s,\tau_i\otimes
\sigma)$ and gamma factors $\gamma(s,\tau_i\otimes \sigma,\psi)$ and
their multiplicativities, it follows that
\[
\gamma(s,\tau_1\otimes \rho,\psi)=\gamma(s,\tau_2\otimes \rho,\psi).
\]
By the $(2r+1,2r)$-local converse theorem in \cite{Hen}, we then
conclude that $\tau_1\cong\tau_2$. Now apply Bernstein-Zelevinsky's
classification of irreducible admissible representations of $G_n$ in
terms of segments, for example Theorem 6.1 in \cite{Z}, we can
conclude that $\pi_1\cong \pi_2$, which finishes the proof.

\end{proof}

\begin{thm}
\label{thm4.5} \textbf{Conjecture 1} is true.
\begin{proof}
This follows from \textbf{Theorem \ref{thm4.3}, \ref{thm4.4}} and
the work \cite{JNS}.
\end{proof}
\end{thm}

\bigskip
\footnotesize \noindent\textit{Acknowledgments.} The author is
grateful to Professor James W.Cogdell for carefully reading the
draft and many helpful suggestions which improve both the
mathematics and exposition of this paper. The author also would like
to thank Professor E.M.Baruch for helpful discussions. We also thank
Baiying Liu and the anonymous referee for pointing out an error in a
previous version of this paper. Part of this work was done during
the author's visit to Morningside Center of Mathematics in January,
2015. The author would like to thank Professor Tian,Ye for his kind
invitation and the hospitality of the center. This work is supported
by the National Natural Science Foundation of China grant 11401193.


\begin{thebibliography}{SK}

%% Use the widest label as parameter.

%% Reference items may be numbered or have labels of your choice.
%% The author's surname PRECEDES the initial of the first name
%% The issue number is only given when the issues are paginated separately.
%% In book titles, first letters are capitalized.
%% Only journal volume numbers are boldfaced.

%%%%%%%%%%% To ease editing, use normal size:

\normalsize
\baselineskip=17pt

%%%%%%%%%%%%%

\bibitem[ALSX]{ALSX}
Moshe Adrian, Baiying Liu, Shaun Stevens and Peng Xu.: On the
Jacquet conjecture on the local converse problem for p-adic $GL_n$.
To appear in Representation Theory,2015. Arxiv:1409.4790.

%\bibitem{Baruch:1997}
%Baruch, E.M. ``On Bessel distributions of $GL_2$ over a $p$-adic
%field.'' \textit{Journal of Number Theory} 67, no.2 (1997): 190-202.

%\bibitem{Baruch:2001}
%Baruch, E.M. ``On Bessel distributions for quasi-split groups.''
%\textit{Transactions of the American Mathematical Society} 353, no.
%7 (2001): 2601-2614.

%\bibitem{Baruch:2003}
%Baruch, E.M. ``Bessel functions for $GL(3)$ over a $p$-adic field.''
%\textit{Pacific Journal of Math.} 211, no. 1 (2003): 1-33.



\bibitem[B95]{Ba1995}
Baruch, E.M.: On local factors attached to representations of p-adic
groups and strong multiplicity one. Ph.D.Thesis, Yale
University(1995).

\bibitem[B97]{Ba1997}
Baruch, E.M.: On the gamma factors attached to representations of
$U(2,1)$ over a p-adic field. Israel Journal of Mathematics.
\textbf{102}, 317-345 (1997)

%\bibitem{Baruch:2004}
%Baruch, E.M. ``Bessel distributions for $GL(3)$ over a p-adic
%field.'' \textit{Pacific Journal of Math.} 217, no. 1 (2004): 11-27.


\bibitem[B05]{Baruch:2005}
Baruch, E.M.: Bessel functions for $GL(n)$ over a p-adic field.
Automorphic representations, L-functions and applications: progress
and prospects, Ohio State Univ. Math. Res. Inst. Publ., 11, de
Gruyter, Berlin,  1--40 (2005).

%\bibitem{BaruchMao:2003}
%Baruch, E.M. and Mao, Zhengyu. ``Bessel identities in the
%Waldspurger correspondence over a $p$-adic field.'' \textit{American
%Journal of Mathematics} 125, no. 2 (2003): 225-288.

%\bibitem{BaruchMao:2005}
%Baruch, E.M. and Mao, Zhengyu. ``Bessel identities in the
%Waldspurger correspondence over the real numbers.'' \textit{Israel
%Journal of Mathematics} 145, (2005): 1-81.

%\bibitem{BaruchMao:2007}
%Baruch, E.M. and Mao, Zhengyu. ``Central value of automorphic
%$L$-functions.'' \textit{Geometric and Functional Analysis} 17, no.
%2 (2007): 333-384.

%\bibitem{Bernstein:1984}
%Bernstein, J.N. ``$P$ invariant distributions on $GL(n)$ and the
%classification of unitary representations of $GL(n)$
%(non-archimedean case).'' \textit{Lie group representations, II}
%Springer, Berlin, 50--102, 1984.

\bibitem[BZ]{BZ}
I.N.Bernstein and A.V.Zelevinsky.: Induced representations of
reductive p-adic groups.I. Ann.Sci.Ecole Norm.Sup.(4) \textbf{10},
no.4, 441-472 (1977)

\bibitem[Chai15]{cjs}
Jingsong Chai.: A weak kernel formula for Bessel functions. Accepted
by Tran.A.M.S.2015. Arxiv:1512.02365.

\bibitem[Chai16]{cjs16}
Jingsong Chai.: Local integrability of Bessel functions for split
groups. Arxiv:1610.03616.


\bibitem[Ch06]{Ch}
Jiang-Ping Jeff Chen.: The $n\times (n-2)$ local converse theorem
for $GL(n)$ over a p-adic field. J. Number Theory \textbf{120},
no.2, 193-205 (2006)


\bibitem[CPS98]{CPS98}
J.W.Cogdell and Piatetski-Shapiro.: Stability of gamma factors for
$SO(2n+1)$. Manuscripta Math. \textbf{95}, no.4, 437-461 (1998)


\bibitem[CPS99]{CPS}
J.W.Cogdell and Piatetski-Shapiro.: Converse theorems for $GL_n$,
II. J.Reine Angew. Math. \textbf{507}, 165-188 (1999)

\bibitem[CPS]{CPS2}
J.W.Cogdell and Piatetski-Shapiro.: Derivatives and L-functions for
$GL(n)$. To appear in \textit{Representation Theory, Number Theory,
and Invariant Theory: In Honor of Roger Howe on the occasion of his
$70^{th}$ Birthday.}. Available at
https://people.math.osu.edu/cogdell.1/


%\bibitem{CKPSSH:2004}
%Cogdell, J.W., Kim,H., Piatetski-Shapiro, I.I. and Shahidi,F.
%``Functioriality for the classical groups.'' \textit{Publ. Math.
%Inst. Hautes Etudes Sci.} no. 99 (2004): 163-233.

\bibitem[CPSS05]{CPSSH:2005}
Cogdell, J.W., Piatetski-Shapiro, I.I. and Shahidi,F.: Partial
Bessel functions for quasi-split groups. Automorphic
representations, L-functions and applications: progress and
properties, 95-128 (2005)


\bibitem[CPSS08]{CPSSH:2008}
Cogdell, J.W., Piatetski-Shapiro, I.I. and Shahidi,F.: Stability of
$\gamma$-factors for quasi-split groups. J. Inst. Math. Jussieu.
\textbf{7}, no.1, 27-66 (2008)


%\bibitem[FLO12]{FLO}
%B.Feigon, E.Lapid and O.Offen.: On representations distinguished by
%unitary groups. Publ.Math.Inst.Hautes \'{E}tudes Sci. \textbf{115},
%185-323 (2012)

%\bibitem{GelfandKazhdan:1975}
%Gelfand, I.M. and Kazhdan, D. ``Representations of the group
%$GL(n,K)$ where $K$ is a local field.'' \textit{Groups and their
%representations},  95--118, Halsted, New York, 1975.


\bibitem[Go]{Goldfeld}
Dorian Goldfeld.: Automorphic forms and L-functions for the group
$GL(n,\mathbb{R})$. Cambridge Studies in Advanced Mathematics, 99.
Cambridge University Press, Cambridge (2006)


\bibitem[He]{Hen}
Guy Henniart.: Caracterisation de la correspondance de Langlands
locale par les facteurs $\epsilon$ de paires. Invent. Math.
\textbf{113}, no.2, 339-350 (1993)


\bibitem[JLa]{JaL}
H.Jacquet and R.Langlands.: Automorphic forms on $GL(2)$. Lecture
Notes in Mathematics, Vol.114. Spinger-Verlag, Berlin-New York
(1970)

\bibitem[JL]{JL}
H.Jacquet and Baiying Liu.: On the Local Converse Theorem for p-adic
GLn. Arxiv:1601.03656.

\bibitem[JPSS79]{JPSS1979}
H. Jacquet, I.I. Piatetski-Shapiro and J. Shalika.: Automorphic
forms on $GL(3)$. Ann.of Math.(2). \textbf{109}, no.1-2, 169-258
(1979)

\bibitem[JPSS83]{JPSS}
H. Jacquet, I.I. Piatetski-Shapiro and J. Shalika.: Rankin-Selberg
convolutions. Amer. J. Math. \textbf{105}, 367-464 (1983)

\bibitem[Jiang]{Jiang}
Dihua Jiang.: On local $\gamma$-factors.  Arithmetic Geometry and
Number Theory. Series on Number Theory and Applications, Vol 1.
World Scientific. 1-28 (2006)

\bibitem[JN]{JN}
Dihua Jiang and Chufeng Nien.: On the local Langlands conjecture and
related problems over p-adic local fields. Proceedings to the 6th
International Congress of Chinese Mathematicians. Taipei (2013)

\bibitem[JNS]{JNS}
Dihua Jiang, Chufeng Nien and Shaun Stevens.: Towards the Jacquet
conjecture on the local converse problem for p-adic $GL_n$.
J.Eur.Math.Soc. \textbf{17}, no.4, 991-1007 (2015)

\bibitem[JS]{JS}
Dihua Jiang and David Soudry.: The local converse theorem for
$SO(2n+1)$ and applications.  Annals of Mathematics. \textbf{157},
743-806 (2003)

\bibitem[LM09]{LM}
Lapid, E. and Mao, Zhengyu.: On the asymptotics of Whittaker
functions. Representation Theory. \textbf{13},63-81 (2009)

\bibitem[LM13]{LapidMao:2013}
Lapid, E. and Mao, Zhengyu.: Stability of certain oscillatory
integrals. International Mathematics Research Notices. no.3, 525-547
(2013)

\bibitem[LM15]{LM15}
Lapid, E. and Mao, Zhengyu.: A conjecture on Whittaker-Fourier
coefficients of cusp forms. J.Number Theory. 146, 448-503 (2015)


\bibitem[N]{Nien}
Chufeng Nien.: A proof of finite field analogue of Jacquet's
conjecture. Amer. J. Math. \textbf{136}, no.3, 653-674 (2014)

\bibitem[Sh84]{Sha}
Shahidi, Freydoon.: Fourier transforms of intertwining operators and
Plancherel measures for $GL(n)$. Amer. J. Math. \textbf{106}, 67-111
(1984)

\bibitem[Sh02]{Shahidi:2002}
Shahidi, Freydoon.: Local coefficients as Mellin transforms of
Bessel functions: towards a general stability. International
Mathematics Research Notices. no.39, 2075-2119 (2002)

%\bibitem{Shalika:1974}
%Shalika, J. ``The multiplicity one theorem for $GL_n$.''
%\textit{Ann. of Math.} 100, (1974): 171-193.

\bibitem[S]{Soudry:1984}
Soudry, David.: The $L$ and $\gamma$ factors for generic
representations of $GSp(4,k)\times GL(2,k)$ over a local
non-archimedean field $k$. Duke Math. Journal. \textbf{51}, no.2,
355-394 (1984)

\bibitem[Ts]{Ts}
Tung-Lin Tsai.: Stability of $\gamma-$factors for $GL(R)\times
GL(R)$. Thesis 2011, Purdue University.

\bibitem[Z]{Z}
A.V.Zelevinsky.: Induced representations of reductive p-adic
groups.II.On irreducible representations of $GL(n)$. Ann.Sci.Ecole
Norm.Sup.(4). \textbf{13}, no.2, 165-210 (1980)

\bibitem[Zh1]{Zh1}
Qing Zhang.: A local convere theorem for $U(1,1)$. Arxiv:1508.07062.

\bibitem[Zh2]{Zh2}
Qing Zhang.: A local converse theorem for $U(2,2)$.Arxiv:1509.00900.


\end{thebibliography}
\end{document}